\documentclass[12pt,psamsfonts]{amsart}
\usepackage{amsmath}
\usepackage{yhmath}
\usepackage{ booktabs}
\usepackage{graphicx,epstopdf,subfigure}
\usepackage{amsthm}
\usepackage{amssymb}
\usepackage{amscd}
\usepackage{amsfonts}
\usepackage{amsbsy}
\usepackage[dvips]{psfrag}

\usepackage{url}
\usepackage{epsfig}
\usepackage{latexsym}
\usepackage[mathscr]{eucal}
\usepackage{graphics}
\usepackage{multirow}
\usepackage{color}
\usepackage{calc}
\usepackage{xcolor}

\allowdisplaybreaks \allowdisplaybreaks[4]

\newcommand{\tabincell}[2]{\begin{tabular}{@{}#1@{}}#2\end{tabular}}

\newcommand{\w}\widetilde

\newtheorem {theorem} {Theorem} 
\newtheorem {proposition} [theorem] {Proposition}
\newtheorem {corollary} [theorem] {Corollary}

\newtheorem {lemma} [theorem] {Lemma}
\newtheorem {definition} {Definition}
\newtheorem {remark} {Remark}

\begin{document}
	
	\title[
Generalized polynomial Li\'enard system with   global centers]
	{ A sufficient and necessary condition of  generalized polynomial Li\'enard systems with global centers} 
	
	\author[H. Chen et al. ]
	{
		Hebai Chen$^{1}$,  Zhijie Li$^{1}$,  Rui Zhang$^{1}$
	}
	
	\address{$^1$
		School of Mathematics and Statistics, HNP-LAMA, Central South University,
		Changsha, Hunan 410083, P. R. China
	}
	\email{chen\_hebai@csu.edu.cn (Chen),
		li\_zhijie@csu.edu.cn (Li),
		zhang\_rui@csu.edu.cn (Zhang).
	}
	
	
			\begin{abstract}
			The aim of this paper is to give a sufficient and necessary condition of the generalized polynomial Li\'enard system with a global center (including
			linear typer and nilpotent type).
			Recently, Llibre and Valls [{\it J. Differential Equations}, {\bf 330} (2022), 66-80] gave a sufficient and necessary condition of the generalized polynomial Li\'enard system with a  linear type global center. It is easy to see that our sufficient and necessary condition is more easy
			by comparison.
			In particular, we provide the explicit expressions of all the generalized polynomial Li\'enard differential systems of degree $5$ having a global center at the origin
			and the explicit expression of a generalized polynomial Li\'enard differential system of  indefinite degree having a global center at the origin.
		\end{abstract}
		
	\maketitle
	\section{ Introduction and main results }
	
	In the beginning, we refer	
	to  \cite{LV}
	and give the following definitions on centers and global centers.
	\begin{definition}
		An equilibrium $p$ is a center  of a planar differential equation if there is 
		a neighbourhood $U$ of $p$  is full of closed orbits.
	\end{definition}
 Notice that	the notion of center traces back to the investigations of Poincar\'e \cite{Poin} and Dulac \cite{Dulac}.

	Let the period annulus of the center $p$ be the maximal connected set of periodic orbits surrounding the center $p$ and having $p$ in its boundary.
	\begin{definition}
		$p$ is a global center if and only if its period annulus is $\mathbb{R}^2 \setminus \{q\}$.
	\end{definition}
		As introduced in  \cite{LV}, some researchers began to  study global centers in the 1990s, see \cite{Conti,Conti98,Gale,Gar,Gar2,HLX,ZLL}.

Many mathematicans such as Smale \cite{Smale91,Smale} and Lins et. al. \cite{LMP} were interested in 	
	the following polynomial Li\'enard system
		\begin{equation}
	\begin{cases}
	\dot{x}=y,		\\
	\dot{y}=-g(x)-f(x)y\\
	\end{cases}
	\label{1}
	\end{equation} 
since system
\eqref{1} has widely real world applications (see \cite{NB}) and  important theoretical consequences (see \cite{CLW,ZDHD}),
	where 
	$$
	g(x)=\sum_{i=r}^{m}a_ix^i,\qquad
	f(x)=\sum_{i=s}^{n}b_ix^i
	$$
	with $m, n, r, s\in\mathbb{N}$  and  $a_ra_mb_sb_n\ne0$, and the dot  represents the derivative of the independent variable $t$.
	Recently,  Llibre and Valls \cite{LV} gave a sufficient and necessary condition of the polynomial Li\'enard system
\eqref{1}
	with a  linear type global center.
	Moreover,
	Llibre and Valls \cite{LV} provided the explicit expressions of all the generalized polynomial Li\'enard differential systems of degree $3$ having a global center at the origin, and the explicit expression of a generalized polynomial Li\'enard differential system of degree $5$ having a global center at the origin.
	In \cite{LV},
	the sufficient and necessary condition of Li\'{e}nard system \eqref{1} with a linear type global center at the origin is 
	that the following conditions hold:
		\begin{itemize}
		\item[\bf (a)] $xg(x)>0$ for all $x\ne0$;
		\item[\bf (b)] There exist real polynomials $h$, $f_1$ and $g_1$ such that
		$$
		f(x)=f_1(h(x))h_1'(x), \  g(x)=g_1(h(x))h_1'(x)
		$$	
		with $h'(0)=0$ and $h''(0)\neq0$; 
		\item[\bf (c)] ${\deg}g=l$ is odd, and ${ \deg}g>1+{ \deg}f$;
		\item[\bf (d)] The local phase portrait of the singular point localized at the origin of the polynomial
		differential system
		\[
		\dot u=uv^{l-1}f\left(\frac{u}{v}\right)-uv^lg\left(\frac{u}{v}\right)-v^{l-1}, \
		\dot v=v^l\left(f\left(\frac{u}{v}\right)-vg\left(\frac{u}{v}\right)\right)	
		\]
		is formed by two hyperbolic sectors.
	\end{itemize}
	However, it's not easy  to verify the conditions {\bf (b)} and {\bf (d)}.
	Notice that the global center can have the  three classifications:
	linear type, nilpotent type and degenerate type,
	where the linear type center is that the Jacobian of the system evaluated at a center has purely imaginary eigenvalues,
	the nilpotent type center is that it has both eigenvalues zero but its linear part is not identically zero,
		the degenerate type center is that it has  its linear part identically zero. 
	Notice that
		the global center of Li\'enard system \eqref{1} cannot be degenerate type since the Jacobian matrix at the origin of Li\'enard system \eqref{1} is
	\begin{eqnarray*}
	\left[\begin{array}{l}
0, \ \   \   \    \   \   \  \   \   \ 1
	\\
-g'(0), ~ -f(0)
	\end{array}\right].
	\end{eqnarray*} 	
	Naturally, based on the results of \cite{LV}, we have the following three questions:
	\begin{itemize}
		\item[\bf (Q1)] Can we give a new sufficient and necessary condition such that 
		the condition can be verified easier?
		\item[\bf (Q2)]  Can we give a  sufficient and necessary condition on the  polynomial  Li\'{e}nard system \eqref{1} with a nilpotent type global center at the origin?
		\item[\bf (Q3)]  Can we provide the explicit expressions of all the generalized polynomial Li\'enard differential systems of degree $5$ having a global center (including linear type and nilpotent type) at the origin?
	\end{itemize}

	In order to answer the question {\bf (Q1)}, we give the following theorem.

	\begin{figure}[hpt]
	\centering
	\includegraphics[scale=0.7]{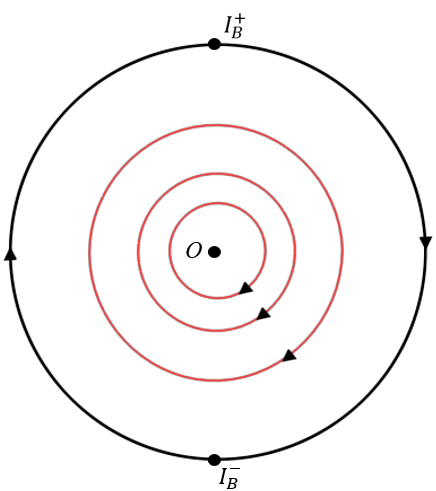}
	\caption{Global center of system \eqref{1}.}
	\label{tu3}
\end{figure}

\begin{theorem}
	\label{thm1}
	System \eqref{1} has a linear type global center at the origin,   if and only if the following conditions hold:
	\begin{itemize}
		\item[\bf (\romannumeral1)]$xg(x)>0$ for all $x\ne0$;
		\item[\bf (\romannumeral2)] $r=1$, $s\geq 1$, $a_r>0$; 
		\item[\bf (\romannumeral3)] $m$ is odd,  $m>2n+1$, $a_m>0$; or  $m=2n+1$ and $4(n+1)a_mb_n^{-2}>1$;
		\item[\bf (\romannumeral4)] $F(x_1)=F(x_2)$ if 
		$G(x_1)=G(x_2)$ for all $x_1<0<x_2$, where $G(x)=\int_{0}^{x}g(\xi)d\xi$.
	\end{itemize}
	Moreover, the global center of system \eqref{1}  is as shown in {\rm Figure \ref{tu3}}.
\end{theorem}
In order to answer the question {\bf (Q2)}, we give the following theorem.

\begin{theorem} 
	\label{thm2}
	System \eqref{1} has a nilpotent type global center at the origin,  if and only if  conditions {\bf (\romannumeral1)},  {\bf (\romannumeral3)},  {\bf (\romannumeral4)} in {\rm Theorem \ref{thm1}} and the following condition hold:	
	\begin{itemize}
		\item[\bf (\romannumeral2$^*$)]$r$ is odd,   $2<r<2s+1$, $a_r>0$;
		or $r=2s+1\geq3$ and $b_s^2-2(r+1)a_r<0$.
	\end{itemize}
	Moreover, the global center of system \eqref{1}  is as shown in {\rm Figure \ref{tu3}}.
\end{theorem}
 \begin{remark}
 	Notice that if condition {\bf (\romannumeral1)} holds, we can naturely obtain that  $a_r>0$, $a_m>0$ and $r, m$ must be odd.  
 	On the one hand,
 for readability, we still write  ``$a_r>0$, $a_m>0$ and $m, r$ is odd" in conditions {\bf (\romannumeral2)-(\romannumeral4)}.
 On the other hand, we will prove that the  condition {\bf (\romannumeral3)} can be equivalent  to 
 that there are no orbits  of system \eqref{1} connecting the equilibria at infinity in the Poincar\'e disc.
 \end{remark}

When $g(x)$ is odd, we can give the following corollary.
\begin{corollary}
	\label{thm5}
	Assume that $g(x)$ is an odd function.
System \eqref{1} having a linear type global center at the origin  if and only if
	conditions {\bf (\romannumeral1)},  {\bf (\romannumeral2)},  {\bf (\romannumeral3)} in {\rm Theorem \ref{thm1}} and
	$f(x)$ is odd;
	polynomial Li\'{e}nard system \eqref{1} having a nilpotent type global center at the origin  if and only if
	conditions {\bf (\romannumeral1)},  {\bf (\romannumeral2$^*$)},  {\bf (\romannumeral3)} in {\rm Theorem \ref{thm1}} and
	$f(x)$ is odd.
\end{corollary}
In order to answer the question {\bf (Q3)}, we give the following theorem.
\begin{theorem}
	\label{thm4}
	All generalized quintic Li\'{e}nard systems having a linear type global center at the origin of coordinates after a rescaling of the variables $x$, $y$ and $t$ can be written as 
	$$
	\dot{x}=y, \quad \dot{y}=-(x+ax^3+x^5)-bxy,
	$$
  where $a>-2$ and $b \neq 0$.
  
	All generalized quintic Li\'{e}nard systems having a nilpotent type global center at the origin of coordinates after a rescaling of the variables $x$, $y$ and $t$ can be written as 
 	$$
 	\dot{x}=y, \quad \dot{y}=-(x^3+x^5)-cxy,
 	$$
  where
  $c \in \left(-2\sqrt{2},0\right)\cup\left(0,2\sqrt{2}\right)$. 
\end{theorem}


The next proposition present a generalized polynomial Li\'enard system of degree $2k+1$ with a linear global center.  

\begin{proposition}
	\label{pro5}
	The following generalized polynomial Li\'{e}nard system of degree $2k+1$
	\begin{equation}
		\label{2m+1}
		\dot{x}=y, \quad \dot{y}=-x-ax^{2k+1}-xy-bx^ly
	\end{equation}
	has a global center at the origin of coordinates if and only if the parameters belong to any one of the following   four parameter spaces 
		\begin{itemize}
		\item[] 
		$\mathcal{S}_1=\{(k,l,a,b)\in\mathbb{N}^2\times\mathbb{R}^2|k>l, l \text{ odd },  a>0, b\ne 0\}$,
	\item[] 
		$\mathcal{S}_2=\{(k,l,a,b)\in\mathbb{N}^2\times\mathbb{R}^2|k=l, l \text{ odd },  a>0, b\ne0, 4(l+1)ab^{-2}>1\}$,
	\item[] 
	$\mathcal{S}_3=\{(k,l,a,b)\in\mathbb{N}^2\times\mathbb{R}^2|k>1,  a>0, b=0\}$,
		\item[] 
		$\mathcal{S}_4=\{(k,l,a,b)\in\mathbb{N}^2\times\mathbb{R}^2|k=1,  a>1/8, b=0\}$.
		\end{itemize}
\end{proposition}

The organization  of the rest of this paper is as follows. To prove Theorems \ref{thm1} and \ref{thm2}, we state a preliminary result in  {\rm Section \ref{pre}}.   Theorem \ref{thm1}, Theorem \ref{thm2} and Corollary \ref{thm5} are proven in Subsections \ref{sub3.1} and \ref{3.2}, respectively.  Moreover,  all   generalized polynomial Li\'enard differential systems of degree $5$ having a global center  are given in Subsection \ref{sec2} by   Theorems \ref{thm1} and \ref{thm2}, 
i.e., Theorem \ref{thm4} is proven.
Furthermore,
we provide the explicit expression  of a  generalized polynomial Li\'enard differential system of indefinite degree $2k+1$ in  in Subsection \ref{3.5}, i.e., Proposition \ref{pro5} is proven. 
   Finally, an interesting result on the boundedness of all orbits of system \eqref{1} is given when 	conditions {\bf (\romannumeral1)},  {\bf (\romannumeral3)}  in {\rm Theorem \ref{thm1}} hold and condition {\bf (\romannumeral4)} does not hold in Section 
\ref{rc}.

\section{Preliminaries}
\label{pre}
To prove the main results, we introduce a preliminary result about closed orbits of a generalized   Li\'enard system in this section.

  With the following global homeomorphism
  transformation
  \[
  (x,y)\to(x,y-F(x)),
  \]
  system \eqref{1} is changed into
  \begin{equation}
  \label{2}
  \left\{\begin{aligned}
  \dot{x}=&y-F(x),
  \\
  \dot{y}=&-g(x),
  \end{aligned}
  \right.
  \end{equation}
    where $F(x):=\int_{0}^{x}f(\xi)d\xi$.
    Here, we require that
  system  \eqref{2} satisfies $xg(x)>0$ for all $x\ne0$.
Define 
\begin{equation}
\label{z}
w(x)=\left((r+1)G(x)\right)^{\frac{1}{r+1}},
\end{equation}
where
$G(x):=\int_{0}^{x}g(\xi)d\xi$. It follows from $xg(x)>0$ for all $x\ne0$ that   
$G(x)\geq0$ and  $w(x)$ is decreasing in $x<0$, increasing in $x>0$. In other words, we have  $w(x)\geq0$. Thus,  let $x_1(w)$ (resp. $x_2(w)$) be  the branch  of the inverse of $w(x)$ for $x\leq0$ (resp. $x\geq0$).
In the zone $x\leq0$, with the transformation $w=w(x)$, system \eqref{2} is changed into 
\begin{equation}
\label{z1}
\frac{dw}{dy}=\frac{w^{-r}g(x)dx}{dy}=\frac{F(x)-y}{w^r}=\frac{F(x_1(w))-y}{w^r}=:\frac{F_1(w)-y}{w^r}, 
\end{equation}
where $w\geq 0$.
Similarly, in the zone $x\geq 0$, the transformation $w=w(x)$ brings system \eqref{2} to  
\begin{equation}
	\label{z2}
\frac{dw}{dy}=\frac{w^{-r}g(x)dx}{dy}=\frac{F(x)-y}{w^r}=\frac{F(x_2(w))-y}{w^r}=:\frac{F_2(w)-y}{w^r},  
\end{equation}
where $w\geq 0$.

Let $\widehat{ABC}$ be any an orbit arc surrounding the origin of system \eqref{2} in $\mathbb{R}^2$,
as shown in Figure \ref{tu1} (a). Moreover, let $\mathcal{G}$ be a set of such  all orbits $\widehat{ABC}$. Define that $\widehat{ABC}\bigcap\{(x,y)|x\leq0\}$   (resp. $\widehat{ABC}\bigcap\{(x,y)|x\geq0\}$)
  of system \eqref{2} correspond to integral curves $\gamma_1$ in $w$-$y$ plane  of equation  \eqref{z1}
  (resp. $\gamma_2$ in $w$-$y$ plane  of equation    \eqref{z2}),   see Figure \ref{tu1}(b). 
  For simplicity, 
we  respectively  refer to the corresponding points  of  $A$, $B$, $C$ in $x$-$y$ plane as $A'$, $B'$, $C'$ in $w$-$y$ plane.
Then, we can check  that $\gamma_1$ (resp. $\gamma_2$) is the integral curve $\widehat{B'C'}$ (resp. $\widehat{B'A'}$) of equation \eqref{z1} (resp. \eqref{z2}) starting from $B'$ and ending at $C'$ (resp. $A'$)  on the $y$-axis of $w$-$y$ plane.
In other words,  $\widehat{ABC}$ is a closed orbit if and only if $A'$ and $C'$  coincide in $w$-$y$ plane. 
\begin{figure}[hpt]
	\centering
\subfigure[{ An orbit arc $\widehat{ABC}$ surrounding the origin }]{
		\includegraphics[width=0.45\textwidth]{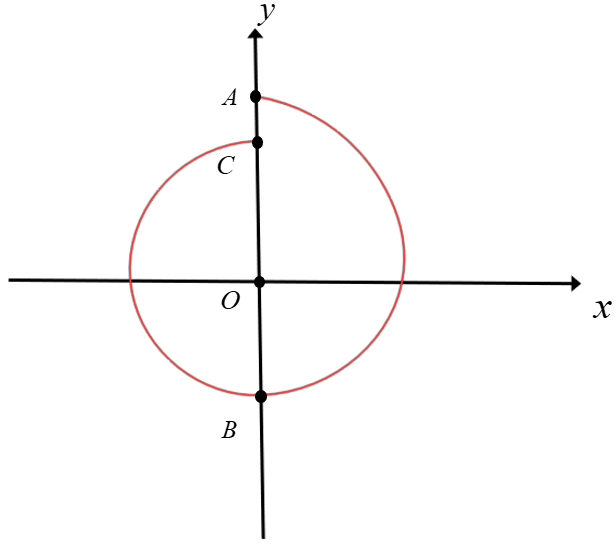}}
	\quad
\subfigure[{ The orbit arcs starting from $B'$
		for equations \eqref{z1} and \eqref{z2} }]{
		\includegraphics[width=0.46\textwidth]{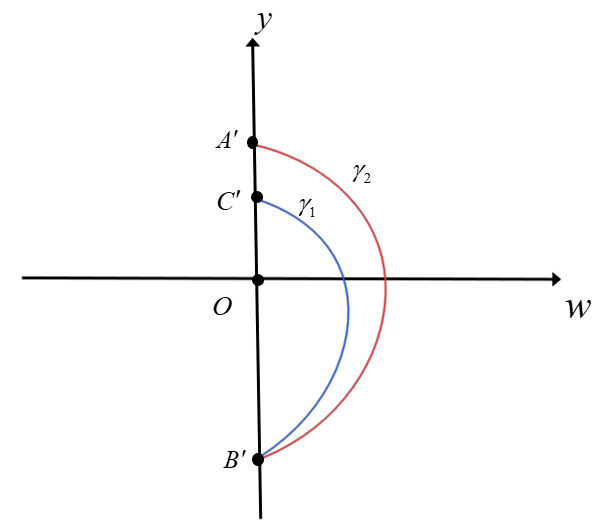}}
	\caption{The orbit arcs for showing the change of transformation \eqref{z} }
	\label{tu1}
\end{figure}

\begin{theorem} 
	\label{thm3}
 The set $\mathcal{G}$ is a closed-orbit set if and only if 
$$
 F_1(w)\equiv F_2(w)
$$
for $0\leq w\leq +\infty$.
\end{theorem}
\begin{proof}
Our proof is clearly divided into the following two steps.

{\noindent{\bf Step 1: 
		Sufficiency.}}
When $F_1(w)\equiv F_2(w)$, it is obvious that equations \eqref{z1} and \eqref{z2} are same. Thus,  $\gamma_1$ and $\gamma_2$, which are integral curves  of equations \eqref{z1} and \eqref{z2} starting from $B'$ respectively,  coincide by  the  uniqueness of solutions. Consequently, $A'$ and $C'$, which are respectively the intersections of $\gamma_1$ and $y$-axis, $\gamma_2$ and $y$-axis, coincide. In other words,   $\widehat{ABC}$ is a closed orbit.
 
{\noindent{\bf Step 2:		Necessity.}}
A straight calculation shows that  
	$$ \frac{dF_i(w)}{dw}=f(x(w))x_i'(w),\quad i=1,2.$$
 According to the definition of $x_i(w)$ $(i=1,2)$, we can obtain that $x_i'(w)$  $(i=1,2)$ are analytic because $g$ is polynomial. Furtherly, we have that $dF_i(w)/dw$ $(i=1,2)$ are analytic. Thus, $F_1(w)$ and $F_2(w)$  are analytic functions.
Assume that $F_1(w)\not\equiv F_2(w)$.
 Thus, we can obtain that $\mu\{w|F_1(w)=F_2(w)\}=0$ since  $F_1(w)$ and $F_2(w)$ are both analytic, where $\mu$ is the measure function. Otherwise,  there exist two  values $w_1$, $w_2$ and a small constant $\varepsilon$ such that $F_1(w)- F_2(w)\equiv0$ for $w\in[w_1,w_2]$, but $F_1(w) -F_2(w)\neq0$ for $w\in (w_1-\varepsilon, w_1)$. Then, we can easily check that there exists $k$ such that either the $k$-th order left derivative of $F_1(w)-F_2(w)$ 
  	at $w_1$  is not zero and $k$-th order right derivative of $F_1(w)-F_2(w)$ 
  	at $w_1$  is zero, which contradicts that $F_1(w)-F_2(w)$ is analytic. In other words, all values satisfying $F_1(w)= F_2(w)$ are isolated when $F_1(w)\not\equiv F_2(w)$. 
  	
  	Assume that $w_3$ is the smallest value such that $F_1(w_3)= F_2(w_3)$ if it exists.  Then, either   $F_1(w)>F_2(w)$ or $F_1(w)<F_2(w)$ holds  for all $w\in(0,w_3)$.
 We only need to consider $F_1(w)<F_2(w)$ for $w\in(0,w_3)$.
 Otherwise,
when the other case $F_1(w)>F_2(w)$ for $w\in(0,w_3)$ holds, we only need to apply a tranformation $(y,t)\to(-y,-t)$ for system \eqref{2}. 
 
 	\begin{figure}[hpt]
 	\centering
 	{ 			\includegraphics[width=0.75\textwidth]{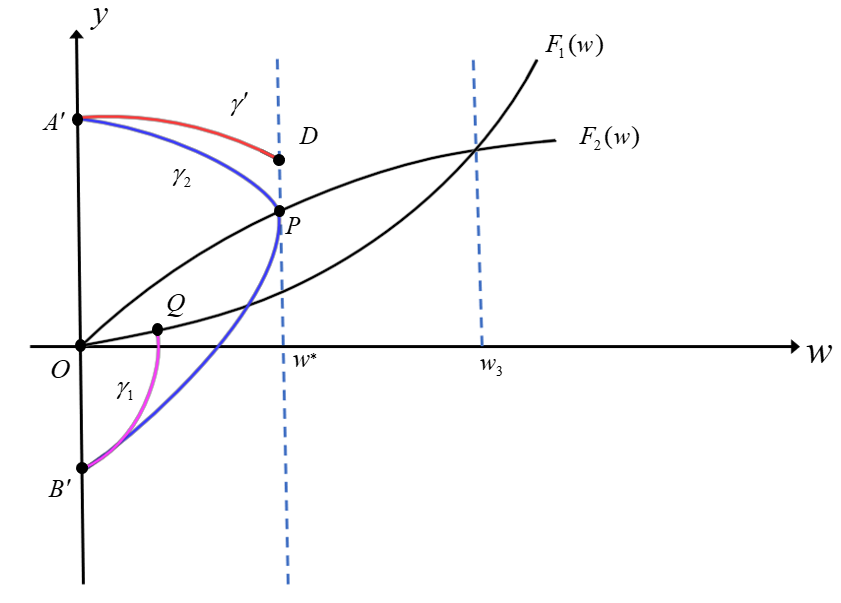}}
 	\caption{$\widehat{ABC}$ is not a closed orbit. }
 	\label{tu2}
 \end{figure}
 
 Choose a point $P:(w^*, F_2(w^*))$ on the curve $y=F_2(w)$  such that $w^*\in(0,w_3)$. We claim that the orbit arc $\widehat{ABC}$  crossing $P$ of system \eqref{2} is not a closed orbit.  Notice that
$\gamma_1$ and $\gamma_2$   are shown  in  
Figure \ref{tu2}, where
$Q:(w_Q,F_2(w_Q))$ is on $y=F_1(w)$. By $F_1(w)<F_2(w)$,
applying  the comparison theorem (see  \cite[Theorem 6.1 of Chapter 1]{Hale}) to  \eqref{z1} and \eqref{z2}, the integral curve $\widehat{B'Q}$ of \eqref{z1} lies on the left-hand side of the integral curve $\widehat{B'P}$ of \eqref{z2}. Thus, $w_Q<w^*$.

On the other hand,  
we define $\gamma'$ to be the integral curve of equation \eqref{z1} starting from $A'$, as shown in  
 Figure \ref{tu2}.   Using the comparison theorem again to  \eqref{z1} and \eqref{z2}, we can obtain that the integral curve $\gamma'\bigcap\{(w,y)|w\in(0,w_3)\}$ 
of \eqref{z1} lies on the right-hand side of the integral curve $\widehat{A'P}$ of \eqref{z2}. Thus, $\gamma'$ must intersect the line $w=w^*$ at $D: (w^*, y_D)$.
 Moreover,  $Q$ is the right-most point on $\gamma_1$ since $y = F_1(w)$ is the vertical isoclinic of equation \eqref{z1}. Therefore, $D$ does not lie on 
 $\gamma_1$ because of $w_Q<w*$, which means that  $\gamma_1$ and $\gamma'$ do not coincide. Consequently, by uniqueness of solutions,  the intersection point of $\gamma_1$
 and $y$-axis cannot be $A'$, i.e.,
 $A'$ and $C'$ do not coincide. Thus, the assertion is proven and the  sufficiency is done. The proof is finished.
\end{proof}

\section{Proofs of main results}
\label{main}
This section is to give the proofs of {our} main results.

\subsection{Proofs of Theorems \ref{thm1} and \ref{thm2}}
\label{sub3.1}
In this subsection, we will give the proofs of Theorems \ref{thm1} and \ref{thm2}. Before the proofs begin,  some lemmas are presented to show the necessity of system \eqref{1} with a global center. 
Notice that if system \eqref{1} has a global center at the origin, there is a unique equilibrium $O:(0,0)$ at finity. 

The first lemma is characterized to obtain  $O$ is the  unique equilibrium. Moreover, we can  roughly determine the qualitative property of $O$ by this lemma.

\begin{lemma}
	\label{lem5}
	System \eqref{1} has a unique equilibrium $O:(0,0)$ which is an anti-saddle if and only if the statement  {\bf(\romannumeral1)} of {\rm Theorem \ref{thm1}}
	holds.
\end{lemma}
\begin{proof}
	On the one hand, if the statement  {\bf(\romannumeral1)} of {\rm Theorem \ref{thm1}}
	holds,
	it is easy to check that $O$ is the  unique equilibrium. By \cite[Theorems 4.1 and 4.2]{ZDHD}, the index of $O$ is 1, which implies that $O$ is an anti-saddle.
	On the other hand, if system \eqref{1} has a unique equilibrium at origin,  $g(x)$ for $x\ne0$ has the following four cases:
	\begin{eqnarray*}
 (C1): g(x)>0; \ (C2): g(x)<0; \ (C3): xg(x)<0; \  (C4): xg(x)>0.
	 \end{eqnarray*}
	Further, by \cite[Theorems 4.1 and 4.2]{ZDHD}, we calculate that the the index of $O$ is $0$ for cases $(C1, C2)$  and $-1$ for case $(C3)$, which contradicts $O$ is an anti-saddle, whose index is $1$.   Thus, statement  {\bf(\romannumeral1)} of {\rm Theorem \ref{thm1}}
must	hold.  The proof is done.

\end{proof}

Secondly, we apply classical  qualitative theory, which can be seen in \cite{Andreev}, \cite[Chapter 3]{DLA},  \cite[Chapter 2]{ZDHD} and other monographs and papers, to state a lemma of the  analysis of  $O$ of system \eqref{1}.

\begin{lemma}
	\label{lem6}
	$O$ is either a linear type center or focus if and only if the statement  {\bf(\romannumeral2)} of {\rm Theorem \ref{thm1}}
	holds, a  nilpotent type center or focus if and only if the statement {\bf(\romannumeral2$^*$)} of {\rm Theorem \ref{thm2}}
	holds.
\end{lemma}
\begin{proof}
	We calculate that the Jacobian matrix at $O$ as follows 
	\begin{equation*}
		\begin{aligned}
			J:=  \begin{pmatrix}
				0 & 1  \\
				-a_1 & -b_0
			\end{pmatrix}.
		\end{aligned}
	\end{equation*}
	The eigenvalues of $J$ are	$\lambda_1=\left(-b_0+\sqrt{b_0^2-4a_1}\right)\left/\right.2$ and $\lambda_2=\left(-b_0-\sqrt{b_0^2-4a_1}\right)\left/\right.2$. 
	When statement  {\bf(\romannumeral2)} of {\rm Theorem \ref{thm1}} holds, we can obtain that $a_1>0$ and $b_0=0$. 
	Thus,  $\lambda_1$ and $\lambda_2$ are a pair of purely imaginary eigenvalues.
By \cite[Theorem 5.1 of Chapter 2]{ZDHD}, $O$ is a linear  type center or focus.

	When statement  {\bf(\romannumeral2$^*$)} of {\rm Theorem \ref{thm2}} holds, $a_1=b_0=0$.  Thus, $\lambda_1=\lambda_2=0$.  However, notice that  not all the coefficients of the
	linear system are zero,  so  $O$ is a nilpotent equilibrium in this case. Furtherly we can obtain that $O$ is a nilpotent  type center or focus  by applying \cite[Theorem 7.2 of Chapter 2]{ZDHD}. We remark that statements  {\bf(\romannumeral2)} of {\rm Theorem \ref{thm1}} and   {\bf(\romannumeral2$^*$)} of {\rm Theorem \ref{thm2}} are sufficient and necessary by \cite[Theorem 5.1 and  Theorem 7.2 of Chapter 2]{ZDHD}.
\end{proof}

Thirdly, as shown in Figure \ref{tu3},  in the Poincar\'e disc, except for those orbits defined by $z=0$, there are no other orbits of system \eqref{1} connecting the equilibria at  infinity.  The following result   provides a criterion on this.

\begin{lemma}
	\label{lem7}
	There is no orbit of system \eqref{1} connecting equilibria at infinity in the Poincar\'{e} disc if and only if the statement  {\bf(\romannumeral3)} of {\rm Theorem \ref{thm1}}
	holds. Moreover, system \eqref{1} at infinity  has two equilibria $I_B^{\pm}$ on the $y$-axis and qualitative propertis of system \eqref{1} near infinity are  shown in {\rm Figure \ref{tu7}(f)}.
\end{lemma}

\begin{proof}
	Without loss of generality, we assume that $b_n>0$ of system \eqref{1}. In fact, if $b_n<0$, we can get a new $b_n>0$ via the transformation $(x,y,t,b_n)\to (x,-y,-t,-b_n)$. 
	
	When  $m\ne2n+1$ is odd, 
	with a coordinate transformation
	$$
	(x,y,t)\rightarrow\left(\left|a_m\right|^{\frac{1}{2n+1-m}}b_n^{-\frac{2}{2n+1-m}}x, \left|a_m\right|^{\frac{n+1}{2n+1-m}}b_n^{-\frac{m+1}{2n+1-m}}y,  \left|a_m\right|^{-\frac{n}{2n+1-m}}b_n^{\frac{m-1}{2n+1-m}}t\right),
	$$
 system \eqref{1} can be rewritten as 
	\begin{equation}
		\label{m}
		\left\{\begin{aligned}
			\dot{x}=&y,
			\\
			\dot{y}=&-\left(\epsilon x^m+\sum_{i=r}^{m-1}\widehat{a}_i{x^i}\right)-y\left(x^n+\sum_{i=s}^{n-1}\widehat{b}_ix^i\right),
		\end{aligned}
		\right.
	\end{equation}
	where $\epsilon={\rm sign}(a_m)$. When $m\ne2n+1$ is even, by the following transformation
	$$
	(x,y,t)\rightarrow\left(a_m^{\frac{1}{2n+1-m}}b_n^{-\frac{2}{2n+1-m}}x, a_m^{\frac{n+1}{2n+1-m}}b_n^{-\frac{m+1}{2n+1-m}}y,  a_m^{-\frac{n}{2n+1-m}}b_n^{\frac{m-1}{2n+1-m}}t\right),
	$$
 system \eqref{1} can be also changed into system \eqref{m}  with $\epsilon=1$. 
 Finally, when $m=2n+1$,
	using a scaling
	$$
	(x,y)\to\left(b_n^{-\frac{1}{n}}x, b_n^{-\frac{1}{n}}y\right),
	$$ 
 we can change system \eqref{1} into system \eqref{m}, where $\epsilon =a_mb_n^{-2}\ne0$.

	For the sake of simplicity, we next only need to consider system \eqref{m}.  According to the result of \cite{DLI}, the dynamics at infinity of system \eqref{m} can be obtained directly, as shown in Table \ref{t}.
	$I_A^{\pm}$ are equilibria on the $x$-axis, $I_B^{\pm}$ are equilibria on the $y$-axis, $I_C^{\pm}$ are equilibria on the line $y=x$ and  $I_D^{\pm}$ are equilibria on the line $y=-x$ where $I_A^+$, $I_B^+$, $I_C^+$, $I_D^+$ lie in the upper half-plane
	and the other four equilibria lie in the lower one. The phase portraits in the Poincar\'{e} disc can be found in Figures \ref{tu4}-\ref{tu8},
	where Figure \ref{tu7}(f) and Figure \ref{tu8}(b) are same. We can obtain that, only in the cases 
	$$
	m=2n+1, \epsilon>(4n+4)^{-1}
	$$ 
	and 
	$$
	 m>2n+1,~ m ~ {\rm odd}, ~\epsilon=1,
	 $$
	   system\eqref{1} has no orbit connecting equilibria at infinity in the Poincar\'{e} disc.  The proof is completed.
\end{proof}

	\begin{remark}
	Notice that	there are no detailed calculations about equilibria at infinity in \cite{DLI}. For the completeness and readability of the article, we will show the detailed calculations about the qualitative properties of system \eqref{m} near infinity in the two cases $m=2n+1$, $\epsilon>(4n+4)^{-1}$ and $m>2n+1$, $m$ is odd, $\epsilon=1$ in the  appendix.	
	\end{remark}

\begin{table}[htp]\scriptsize
	\centering \scriptsize \doublerulesep 0.5pt
	\renewcommand\baselinestretch{1.7}\selectfont
	\centering
	\caption{Properties of  equilibria at infinity}
	\label{t}
	\resizebox{\textwidth}{!}{
		\begin{tabular}{|c|c|c|l|}
			\hline
			\multicolumn{3}{|c|}{possibilities of ($m$, $n$, $\epsilon$)} & \multicolumn{1}{c|}{types}\\
			\hline
			\multirow{13}{*}{$m<n+1$} &\multicolumn{2}{c|}{$m$, $n$ even} &\tabincell{l}{$I_A^{+}$: saddle, $I_A^{-}$: stable node;\\ $I_B^{\pm}$: unstable degenerate nodes (shown in Figure \ref{tu4}(a)) }\\
			\cline{2-4}
			&\multirow{2}{*}{$m$ odd, $n$ even} &$\epsilon=1$  &  $I_A^{\pm }$: saddles; $I_B^{\pm }$:  unstable degenerate nodes (shown in Figure \ref{tu4}(b))\\
			\cline{3-4}
			& &$\epsilon=-1$   &  $I_A^{\pm }$: stable  nodes; $I_B^{\pm }$: unstable degenerate nodes (shown in Figure \ref{tu4}(c))\\
			\cline{2-4}
			& \multicolumn{2}{c|}{$m$ even, $n$ odd} & \tabincell{l}{$I_A^{+}$:  saddle, $I_A^{-}$: unstable node; \\$S_{\delta}(I_B^+)$ consists of one hyperbolic sector, \\
				$S_{\delta}(I_B^-)$ consists of one  elliptic sector  (shown in Figure \ref{tu4}(d)) }  \\
			\cline{2-4}
			&\multirow{4}{*}{$m$, $n$ odd} &$\epsilon=1$  & \tabincell{l}{$I_A^{\pm}$: saddles ; \\$S_{\delta}(I_B^+)$ consists of one hyperbolic sector, \\
				$S_{\delta}(I_B^-)$ consists of one  elliptic sector (shown in Figure \ref{tu4}(e)) }\\
			\cline{3-4}
			& &$\epsilon=-1$   & \tabincell{l}{$I_A^{+}$: stable node; $I_A^{-}$: unstable node;\\
				$S_{\delta}(I_B^+)$ consists of one hyperbolic sector, \\
				$S_{\delta}(I_B^-)$ consists of one  elliptic sector (shown in Figure \ref{tu4}(f))} \\
			\hline
			\multirow{4}{*}{$m=n+1$}  &\multirow{2}{*}{$n$ even} &$\epsilon=1$  &$I_B^{\pm}$: unstable  degenerate nodes; $I_D^{\pm}$ saddles (shown in Figure \ref{tu5}(a))\\
			\cline{3-4}
			& &$\epsilon=-1$  & $I_B^{\pm }$: unstable degenerate  nodes; $I_C^{\pm }$ stable nodes (shown in Figure \ref{tu5}(b))\\
			\cline{2-4}
			&\multicolumn{2}{c|}{$n$ odd} &\tabincell{l}{$S_{\delta}(I_B^+)$ consists of one hyperbolic sector,
				$S_{\delta}(I_B^-)$ consists of one  elliptic sector; \\ $I_D^{+}$: saddle, $I_D^{-}$: unstable  node (shown in Figure \ref{tu5}(c))}\\
			\hline
			\multirow{11}{*}{$n+1<m<2n+1$} &\multicolumn{2}{c|}{$m$, $n$ even}  &\tabincell{l}{$I_B^+$ unstable degenerate node, \\$S_{\delta}(I_B^-)$ consists of one elliptic sector and one hyperbolic sector,   \\(shown in Figure \ref{tu6}(a))} \\
			\cline{2-4}
			&\multirow{2}{*}{$m$ odd, $n$ even} &$\epsilon=1$ &\tabincell{l}{$I_B^{\pm}$: saddle-nodes (shown in Figure \ref{tu6}(b))} \\
			\cline{3-4}
			& &$\epsilon=-1$ &\tabincell{l}{$S_{\delta}(I_B^{\pm})$ consist of one  elliptic sector respectively (shown in Figure \ref{tu6}(c))} \\
			\cline{2-4}
			&\multicolumn{2}{c|}{$m$ even, $n$ odd}  &\tabincell{l}{  $I_B^+$ unstable degenerate  node, \\$S_{\delta}(I_B^+)$ consists of one  elliptic sector and one hyperbolic sector, \\(shown in Figure \ref{tu6}(d))} \\
			\cline{2-4}
			&\multirow{2}{*}{$m$, $n$ odd} &$\epsilon=1$ &\tabincell{l}{$S_{\delta}(I_B^+)$ consists of one  hyperbolic sector, \\ $S_{\delta}(I_B^-)$ consists of one elliptic sector and  two  hyperbolic sectors  \\(shown in Figure \ref{tu6}(e))} \\
			\cline{3-4}
			&  & $\epsilon=-1$ &\tabincell{l}{$S_{\delta}(I_B^{\pm})$ consist of one  elliptic sector respectively (shown in Figure \ref{tu6}(f))}\\
			\hline
			\multirow{10}{*}{$m=2n+1$} &\multicolumn{2}{c|}{$\epsilon<0$} &\tabincell{l}{$S_{\delta}(I_B^{\pm})$ consist of one elliptic sector respectively (shown in Figure \ref{tu7}(a)) }\\
			\cline{2-4}
			&\multirow{2}{*}{$n$ even} &$0<\epsilon<(4n+4)^{-1}$ &\tabincell{l}{	$I_B^{\pm}$: saddle-nodes (shown in Figure \ref{tu7}(b))} \\
			\cline{3-4}
			& & $\epsilon=(4n+4)^{-1}$ &$I_B^{\pm}$: saddle-nodes (shown in Figure \ref{tu7}(c)) \\
			\cline{2-4}
			&\multirow{4}{*}{$n$ odd} &$0<\epsilon<(4n+4)^{-1}$&\tabincell{l}{	 $S_{\delta}(I_B^{+})$ consist of one  hyperbolic sector, \\
			$S_{\delta}(I_B^-)$ consists of one elliptic sector and  two  hyperbolic sectors\\
			 (shown in Figure \ref{tu7}(d))} \\
			\cline{3-4}
			& & $\epsilon=(4n+4)^{-1}$ & \tabincell{l}{	$S_{\delta}(I_B^+)$ consists of one  hyperbolic sector,  \\
			$S_{\delta}(I_B^-)$ consists of one elliptic sector and  two  hyperbolic sectors\\ (shown in Figure \ref{tu7}(e))}\\
			\cline{2-4}
			&\multicolumn{2}{c|}{$\epsilon>(4n+4)^{-1}$} &\tabincell{l}{$S_{\delta}(I_B^{\pm})$ consists of one  hyperbolic sector respectively (shown in Figure \ref{tu7}(f))} \\
			\hline
			\multirow{3}{*}{$m>2n+1$} &\multicolumn{2}{c|}{$m$ even} &\tabincell{l}{$I_B^+$: stable degenerate node,\\ $I_B^-$: unstable degenerate node (shown in Figure \ref{tu8}(a))} \\
			\cline{2-4}
			&\multirow{2}{*}{$m$ odd} &$\epsilon=1$ & $S_{\delta}(I_B^{\pm})$ consist of one  hyperbolic sector respectively (shown in Figure \ref{tu8}(b))\\
			\cline{3-4}
			& &$\epsilon=-1$ & \tabincell{l}{$S_{\delta}(I_B^{\pm})$ consist of one elliptic sector respectively  (shown in Figure \ref{tu8}(c)) }\\
			\hline
		\end{tabular}
	}
	
	\begin{center}
		Remark : $S_\delta(I_B^{+})$(resp. $S_\delta(I_B^{-})$ ) stands for any a small neighborhood of the equilibrium $I_B^+$ (resp. $I_B^-$).
	\end{center}
\end{table}
	\begin{figure}[hpt]
		\centering
		\subfigure[$m$, $n$ even]{
			\includegraphics[width=0.25\textwidth]{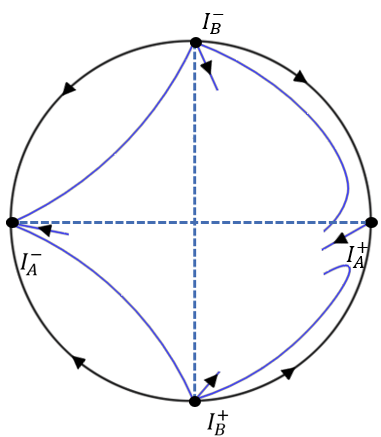}}
		\quad
		\subfigure[$m$ odd, $n$ even, $\epsilon=1$]{
			\includegraphics[width=0.25\textwidth]{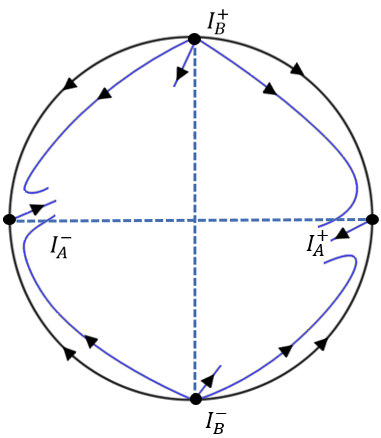}}
		\quad
		\subfigure[$m$ odd, $n$ even, $\epsilon=-1$]{
			\includegraphics[width=0.25\textwidth]{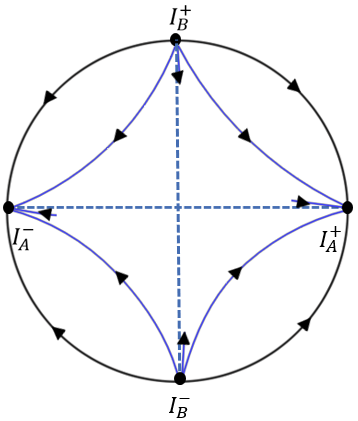}}

		\subfigure[$m$ even, $n$ odd]{
			\includegraphics[width=0.25\textwidth]{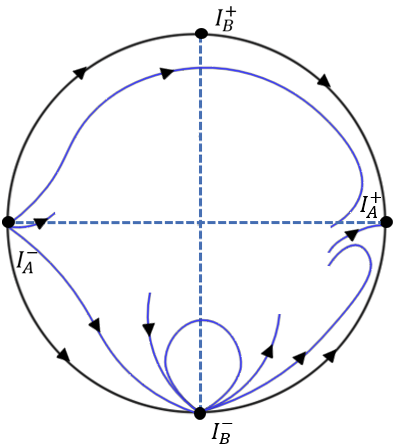}}
		\quad
		\subfigure[$m$, $n$ odd, $\epsilon=1$]{
			\includegraphics[width=0.25\textwidth]{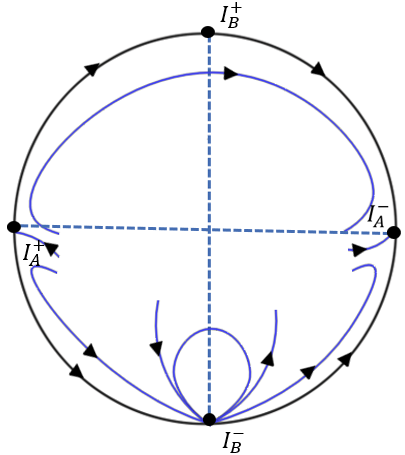}}
		\quad
		\subfigure[$m$, $n$ odd, $\epsilon=-1$]{
			\includegraphics[width=0.25\textwidth]{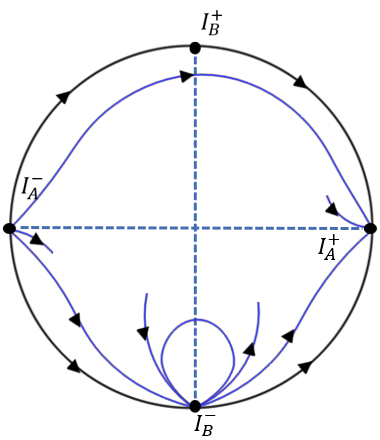}}
		\caption{ Behaviour near infinity in the Poincar\'{e} disc for $m<n+1$.}
		\label{tu4}
	\end{figure}
	
	\begin{figure}[hpt]
		
		\centering
		\subfigure[$n$ even, $\epsilon=1$]{
			\includegraphics[width=0.25\textwidth]{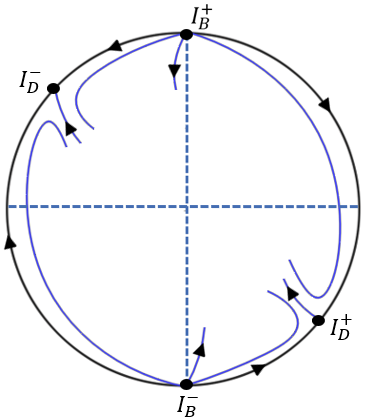}}\quad
		\subfigure[$n$ even, $\epsilon=-1$]{
			\includegraphics[width=0.25\textwidth]{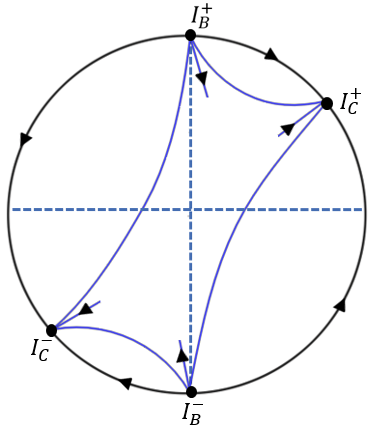}}\quad
		\subfigure[$n$ odd]{
			\includegraphics[width=0.26\textwidth]{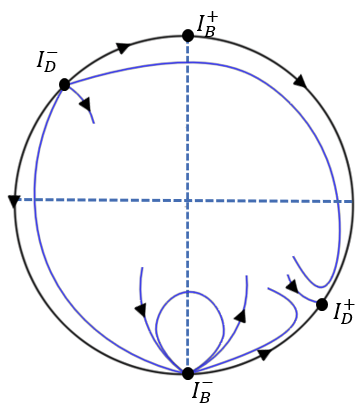}}\quad
		\caption{ Behaviour near infinity in the Poincar\'{e} disc for $m=n+1$.}
		\label{tu5}
	\end{figure}
	
		\begin{figure}[hpt]
			
			\centering
			\subfigure[$m$, $n$ even]{
				\includegraphics[width=0.25\textwidth]{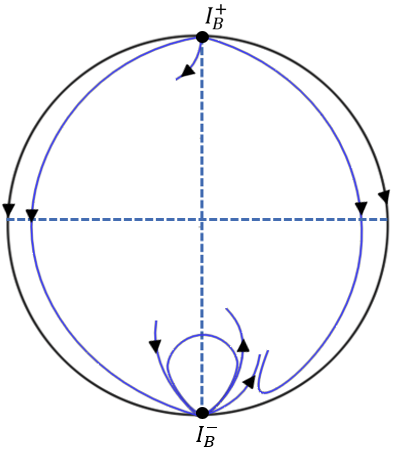}}\quad
			\subfigure[$m$ odd, $n$ even, $\epsilon=1$]{
				\includegraphics[width=0.25\textwidth]{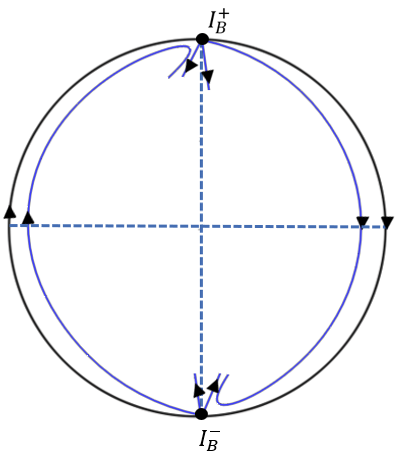}}\quad
			\subfigure[$m$ odd, $n$ even, $\epsilon=-1$]{
				\includegraphics[width=0.25\textwidth]{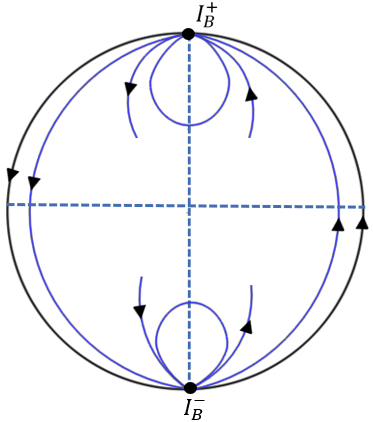}}
			
			\subfigure[$m$ even, $n$ odd]{
				\includegraphics[width=0.25\textwidth]{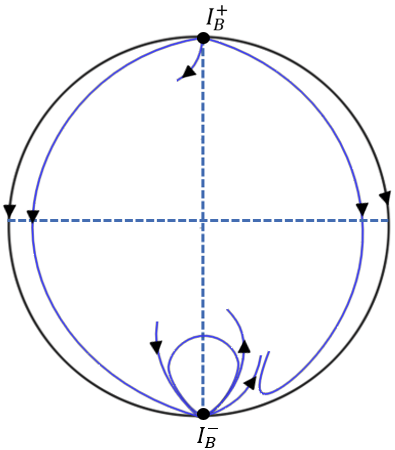}}
			\subfigure[$m$, $n$ odd, $\epsilon=1$]{
				\includegraphics[width=0.25\textwidth]{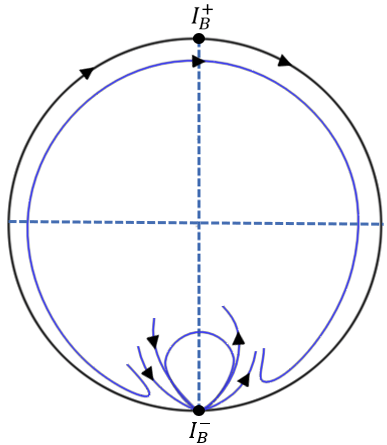}}\quad
			\subfigure[$m$, $n$ odd, $\epsilon=-1$]{
				\includegraphics[width=0.25\textwidth]{58}}
			\caption{ Behaviour near infinity in the Poincar\'{e} disc for $n+1<m<2n+1$.}
			\label{tu6}
		\end{figure}
	\begin{figure}[hpt]
		
		\centering
		\subfigure[$\epsilon<0$]{
			\includegraphics[width=0.25\textwidth]{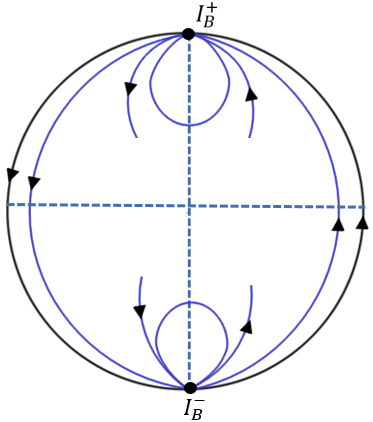}}\quad
		\subfigure[$n$ even, $0<\epsilon<(4n+4)^{-1}$]{
			\includegraphics[width=0.25\textwidth]{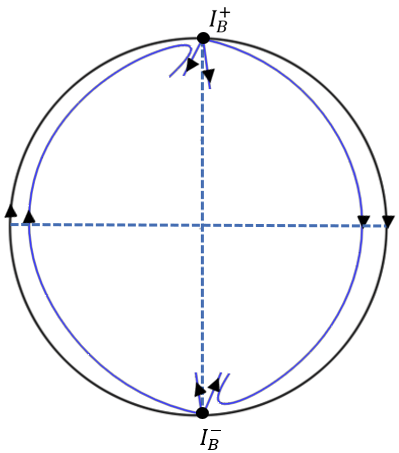}}\quad
			\subfigure[$n$ even, $\epsilon=(4n+4)^{-1}$]{
				\includegraphics[width=0.25\textwidth]{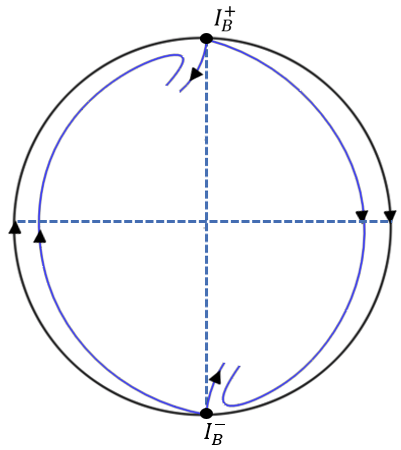}}
			
		\subfigure[$n$ odd, $0<\epsilon<(4n+4)^{-1}$]{
			\includegraphics[width=0.25\textwidth]{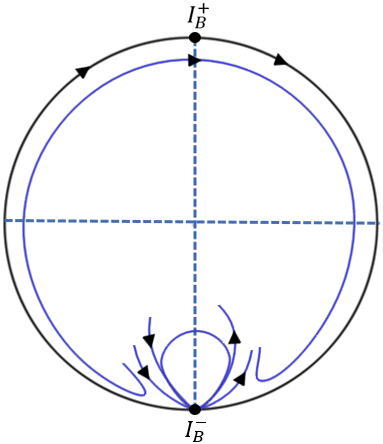}}\quad
		\subfigure[$n$ odd, $\epsilon=(4n+4)^{-1}$]{
			\includegraphics[width=0.25\textwidth]{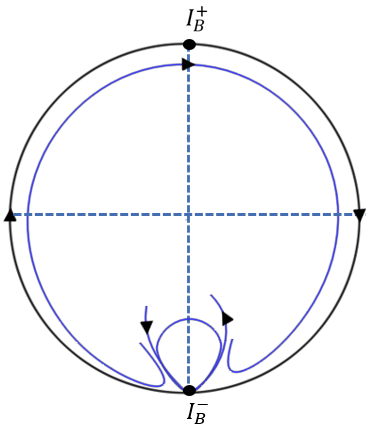}}\quad
		\subfigure[$\epsilon>(4n+4)^{-1}$]{
			\includegraphics[width=0.25\textwidth]{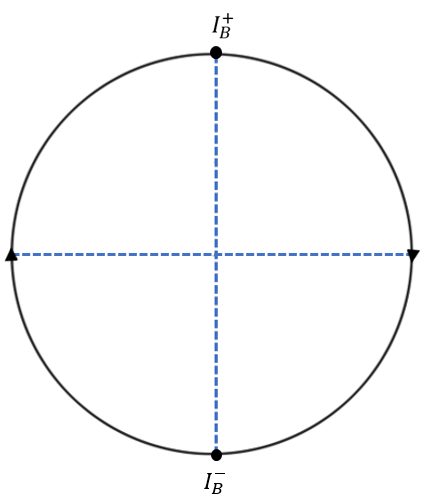}}
		\caption{ Behaviour near infinity in the Poincar\'{e} disc for $m=2n+1$.}
		\label{tu7}
	\end{figure}
	\begin{figure}[hpt]
		
		\centering
		\subfigure[$m$ even]{
			\includegraphics[width=0.25\textwidth]{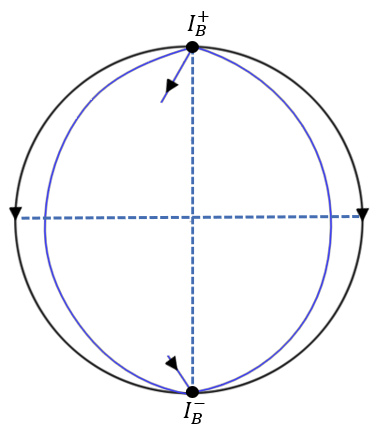}}\quad
		\subfigure[$m$ odd, $\epsilon=1$]{
			\includegraphics[width=0.25\textwidth]{66}}\quad
		\subfigure[$m$ odd, $\epsilon=-1$]{
			\includegraphics[width=0.25\textwidth]{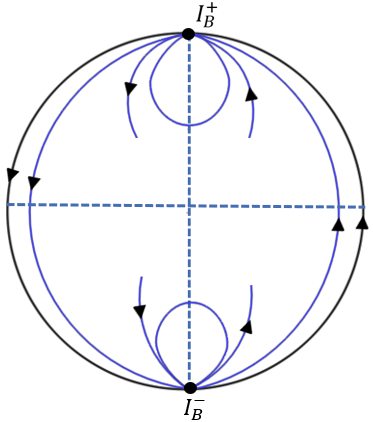}}
		\caption{ Behaviour near infinity in the Poincar\'{e} disc for $m>2n+1$.}
		\label{tu8}
	\end{figure}

{\noindent{\bf Proof of Theorem \ref{thm1}.}
	At first, we will give a proof that system \eqref{1} has a linear type global center at the origin when statements {\bf (\romannumeral1-\romannumeral4)}  hold.
	
	The proof is as follows.
 When statements {\bf (\romannumeral1-\romannumeral2)} of Theorem \ref{thm1}  hold, we can obtain that system \eqref{1} has a unique  equilibrium $O$, which is a linear type center or focus by Lemmas \ref{lem5} and \ref{lem6}.  
	Then, we consider system \eqref{2}, which is  globally topologically equivalent to  system  \eqref{1}. Choose any a point $P$ in $\mathbb{R}^2$. Without loss of generality, assume that $P$ is locate in the first quadrant.  let $\varphi(P,I^+)$  be the positive orbit of system \eqref{2} having the initial point $P$. According to the vector field $(y-F(x),-g(x))$ of system \eqref{2}, it is obvious that $\varphi(P,I^+)$ has to intersect the negative $y$-axis at a first time and then return to
	the positive $y$-axis. Because of arbitrariness of $P$, 
 all orbits of system \eqref{2} belong to orbit-set  $\mathcal{G}$.
As said in Section \ref{pre}, for any a value $w_0\geq0$, we can find a unique group $(x_1,x_2)\in\mathbb{R}^2$ satisfying $x_1<0<x_2$, $G(x_1)=G(x_2)$ and $w(x_1)=w(x_2)=w_0$.
When statement {\bf (\romannumeral4)}  holds,  $F(x_1)=F(x_2)$. Thus, $F(x_1(w_0))=F(x_2(w_0))$. Furtherly, considering definitions of $F_1(w)$ and $F_2(w)$,  we can obtain $F_1(w_0)=F_2(w_0)$. According to the arbitrariness of $w_0$, we can obtain that $F_1(w)\equiv F_2(w)$ for all $w\geq0$ if statement {\bf (\romannumeral4)}  holds.
Naturally,  by Theorem \ref{thm3}, all orbits of system \eqref{2} are closed orbits,  so is  system \eqref{1}. 
Finally by Lemma \ref{lem7}, the  qualitative properties of the equilibria  at infinity  are shown in Figure \ref{tu7}(f)  if  statement {\bf (\romannumeral3)}  holds.  

From what has been discussed above, $O$ is a linear type global center and we can furtherly obtain the global phase portraits of system \eqref{1}, as shown in Figure \ref{tu3}. Thus,   statements {\bf (\romannumeral1-\romannumeral4)}  are sufficient conditions for that system \eqref{1} has a linear type global center at the origin.
The next, we will show that  these statements are also necessary.

Since $O$ is the linear type global center of system \eqref{1}, it is clearly  evident that $O$ is the unique equilibrium which is a center. Thus, by Lemmas \ref{lem5} and \ref{lem6}, statements {\bf (\romannumeral1-\romannumeral2)}  must hold. Moreover, it is easy to check that all orbits are bounded, which implies that 
statement {\bf (\romannumeral3)} holds by Lemma 
\ref{lem7}.  Furtherly, since all orbits are closed orbits,  $F_1(w)\equiv F_2(w)$ for all $0\leq w\leq+\infty$ by applying Theorem \ref{thm3}. 
Return to \eqref{z}, the from of $w$, we can obtain that if $G(x_1)=G(x_2)$ for all $x_1<0<x_2$, $w(x_1)=w(x_2)$ holds. Thus, $F_1(w(x_1))=F_2(w(x_2))$, i.e. $F(x_1)=F(x_2)$. Statement {\bf (\romannumeral4)} holds.
The proof of Theorem \ref{thm1} is done.
$\hfill{} \Box$

{\noindent{\bf Proof of Theorem \ref{thm2}.}
We only need to prove that statement {\bf (\romannumeral2$^*$)} is a sufficient and  necessary condition for that  $O$ of system \eqref{1}  is a nilpotent type center or focus. In fact, it can be obtained by Lemma \ref{lem6}. 
Like the  proof of Theorem \ref{thm1} , we can prove similarly that the condition including   statements {\bf (\romannumeral1)},  {\bf (\romannumeral3)},  {\bf (\romannumeral4)} of Theorem \ref{thm1} and  {\bf (\romannumeral2$^*$)} of Theorem \ref{thm2} is  sufficient and  necessary for     a nilpotent type global center at the origin of system \eqref{1}.
$\hfill{} \Box$

\subsection{Proof of Corollary \ref{thm5}.}
\label{3.2}
Consider that $g(x)$ is odd.
Therefore,
it is clear that $G(x)$ is even. 
Furthermore,  $G(x)$ is strictly increasing for $x>0$ and strictly decreasing for $x<0$.
Thus, we have  $x_1=-x_2$ when $G(x_1)=G(x_2)$ for all $x_1<0<x_2$. 

When $f(x)$ is odd,
we first prove that statement {\bf (\romannumeral4)} of Theorem \ref{thm1}  holds.
Since $f(x)$ is odd,
it is clear that $F(x)$ is even.  Thus,
we have $F(x_1)=F(-x_1)$ for all $x_1<0$.
In other words, statement {\bf (\romannumeral4)} of Theorem \ref{thm1} holds. 

Next, when statement {\bf (\romannumeral4)} of Theorem \ref{thm1}  holds, 
we need to prove that $f(x)$ is odd.
Then, it is obvious that $F(x_1)=F(-x_1)$ for all $x_1<0$
since $F(x_1)=F(-x_1)$ for all $x_1<0$.
Thus, $F(x)$ is even. In other words, $f(x)$ is odd.

In concluions, the condition that $f(x)$ is odd is equivalent to {\bf (\romannumeral4)} of Theorem \ref{thm1}. 
Thus, by Theorem \ref{thm1}, we can obtain the 
conclusion directly.
The proof is finished.
$\hfill{} \Box$

\subsection{Proof of Theorem \ref{thm4}}
\label{sec2}

	Considering quintic Li\'{e}nard system \eqref{1}, we can get $m=5$, $n=4$, or $m=5$, $n<4$, or $m<5$, $n=4$ immediately.
By the statement {\bf (\romannumeral3)} of Theorem \ref{thm1}, $m\geq 2n+1$. If $n=4$, we have $m\geq 9$,
which contradicts $m\leq5$. On the other hand, it follows form {\bf (\romannumeral3)} that $m=5$ and $n\leq2$ furtherly.
Next, we claim that  the value of $n$ must be 1. Now, we consider $n=2$. However,
we can obtain $x_2=O(\left|x_1\right|)$ when $G(x_1)=G(x_2)$ with $\left|x_1\right|$ sufficiently large,  implying that $F(x_1) F(x_2)<0$ since $\deg(F)=3$. 
In other words, $F(x_1)=F(x_2)$ does not hold when $G(x_1)=G(x_2)$  for all $x_1<0<x_2$,
which contradicts statement {\bf (\romannumeral4)}. Then, the assertion is proven. Furthermore, $s\geq1$ by the statement  {\bf (\romannumeral2)} of Theorem \ref{thm1} or  {\bf (\romannumeral2$^*$)} of Theorem \ref{thm2}. Thus, $n=s=1$, i.e., we can let $f(x)=b_1x$ with $b_1\ne0$.

Applying the statement {\bf (\romannumeral4)} of Theorem \ref{thm1} again, $g(x)$ must be odd since $f(x)$ is odd. Moreover, since $f(x)=b_1x$ and $s=1$, we can get that $r=1$ or $r=3$ by the statement  {\bf (\romannumeral2)} of Theorem \ref{thm1} or {\bf (\romannumeral2$^*$)} of Theorem \ref{thm2}. To obtain the concrete form of system \eqref{1}, we distinguish two cases: $r=1$ and $r=3$. In  the first case, the global center is    linear   type, and in the second case, it is nilpotent type.

Case 1: $r=1$. System \eqref{1} has the following form,
\begin{equation}
	\label{quintic}
	\left\{\begin{aligned}
		\dot{x}=&y,
		\\
		\dot{y}=&-(a_1x+a_3x^3+a_5x^5)-b_1xy,
	\end{aligned}
	\right.
\end{equation}
with $a_1a_5b_1\ne0$. Besides, it follows from $xg(x)>0$ for $x\ne 0$ that $a_1>0$ and $a_5>0$. A scaling transformation 
	$$
	(x,y,t)\to\left(a_1^{\frac{1}{4}}a_5^{-\frac{1}{4}}x, a_1^{\frac{3}{4}}a_5^{-\frac{1}{4}}y, a_1^{-\frac{1}{2}}t\right)
	$$	
brings system \eqref{quintic} to 
\begin{equation}
	\label{quintic1}
	\left\{\begin{aligned}
		\dot{x}=&y,
		\\
		\dot{y}=&-(x+ax^3+x^5)-bxy,
	\end{aligned}
	\right.
\end{equation}
where $a=a_1^{-1/2}a_3a_5^{-1/2}$ and $b=b_1a_1^{-1/4}a_5^{-1/4} \ne 0$.
Further, we can obtain that $a>-2$ from $xg(x)=x^2+ax^4+x^6>0$ for $x\ne0$. Thus, system \eqref{quintic1} with $a>-2$ and $b\ne0$  has a linear type global center at the origin. 

Case 2: $r=3$. Rewrite system \eqref{1} as 
\begin{equation}
	\label{tri-quintic}
	\left\{\begin{aligned}
		\dot{x}=&y,
		\\
		\dot{y}=&-(a_3x^3+a_5x^5)-b_1xy,
	\end{aligned}
	\right.
\end{equation}
where $a_3a_5b_1\ne0$. Similarly, we have $a_3>0$ and $a_5>0$.
With a transformation 
	$$
	(x,y,t)\to\left(a_3^{\frac{1}{2}}a_5^{-\frac{1}{2}}x, a_3^{\frac{3}{2}}a_5^{-1}y, a_3^{-1}a_5^{\frac{1}{2}}t\right),
	$$
system \eqref{tri-quintic} is changed into 
\begin{equation}
	\label{tri-quintic1}
	\left\{\begin{aligned}
		\dot{x}=&y,
		\\
		\dot{y}=&-(x^3+x^5)-cxy,
	\end{aligned}
	\right.
\end{equation}
where $c=b_1a_3^{-1/2}\ne0$.
Moreover, by the statement {\bf (\romannumeral2$^*$)} of Theorem \ref{thm2}, $c^2-8<0$ and $c\neq0$, i.e. $c\in\left(-2\sqrt{2},0\right)\cup\left(0,2\sqrt{2}\right)$. Thus,  system \eqref{tri-quintic1} with $c\in\left(-2\sqrt{2},0\right)\cup\left(0,2\sqrt{2}\right)$ has a nilpotent type global center at the origin.
The proof is finished.
$\hfill{} \Box$
\subsection{Proof of Proposition \ref{pro5}.} 
\label{3.5}
Considering Li\'enard system \eqref{2m+1} of degree $2k+1$, let $g(x):=x+ax^{2k+1}$ and $f(x):=x+bx^l$. Thus, $r=1$, $m=2k+1$, $s=1$,   $n=l$ for $b\ne0$, $n=1$ for $b=0$, where $m, r,  n,s$ stand  respectively for the highest and lowest orders of $g(x)$ and $f(x)$. 

First we shall prove that system \eqref{2m+1} has a global center in one of $\mathcal{S}_1,\ldots,\mathcal{S}_4$.
In  these spaces, $xg(x)>0$ for all $x\ne0$ since $a>0$.  
One can check that the statement {\bf (\romannumeral1)} of Theorem \ref{thm1} holds. 
Since $r=1$, $s=1$ and $a>0$,  statement {\bf (\romannumeral2)} of Theorem \ref{thm1} holds.
 Moreover, we can obtain that $g(x)$ and $f(x)$ are odd functions when $(m,n,a,b)\in\mathcal{S}_i$, where $i=1,2,3,4$.   
Then, we will check that statement {\bf (\romannumeral3)} of Theorem \ref{thm1} holds in $\mathcal{S}_1,\ldots,\mathcal{S}_4$. 

Case 1: $(m,n,a,b)\in \mathcal{S}_1$.  In this case, $m=2k+1$ is odd and $n=l$.  Thus, $m>2n+1$ for $k>l$. Then   statement {\bf (\romannumeral3)} of Theorem \ref{thm1} holds since $a>0$ and $k>l$.

Case 2: $(m,n,a,b)\in \mathcal{S}_2$. In this case, $m=2k+1$ is odd,  $n=l$, $a_m=a$ and $b_n=b$. It follows from $k=l$ and $4(n+1)ab^{-2}>1$ that statement {\bf (\romannumeral3)} of Theorem \ref{thm1} holds.

Case 3: $(m,n,a,b)\in \mathcal{S}_3$.  In this case, $m=2k+1$ is odd,  $n=1$. Thus, $2k+1>3$ for $k>1$.   Then  statement {\bf (\romannumeral3)} of Theorem \ref{thm1} holds since $a>0$ and $k>1$.

Case 4: $(m,n,a,b)\in \mathcal{S}_4$. 
In this case $m=3$,  $n=1$, $a_m=a$ and $b_n=1$.
When $a>1/8$, $4(n+1)a_mb_n^{-2}>1$ holds. Consequently, statement {\bf (\romannumeral3)} of Theorem \ref{thm1} holds.

 In summary, statement {\bf (\romannumeral3)} of Theorem \ref{thm1} holds for $(m,n,a,b)\in\mathcal{S}_i$ where $i=1,2,3,4$.  Thus, by Corollary \ref{thm5}, system  \eqref{2m+1} has a linear type global center at the origin.
 
Our  next task now in this proof is to show that the parameters of system \eqref{2m+1}  belong to any one of spaces $\mathcal{S}_1,\ldots,\mathcal{S}_4$ when 
system \eqref{2m+1} has a global center at the origin. Since $g(x)$ is odd,  by Corollary \ref{thm5}, statements {\bf (\romannumeral1)-(\romannumeral3)} of Theorem \ref{thm1} hold and $f(x)$ is an odd function.  

It is easy to check  statement {\bf (\romannumeral2)}  holds for all $(k,l,a,b)\in\mathbb{N}^2\times\mathbb{R}^2$. 
Then we can obtain $a\geq0$ from statement {\bf (\romannumeral1)}. Moreover, $a\ne0$ because the degree of system \eqref{2m+1} is $2k+1$. Thus, $a>0$.  Consider $b\ne0$. In this case, $l$ is odd since $f(x)$ is an odd function.
 It follows from  statement {\bf (\romannumeral3)} that either  $k>l$ or $k=l$, $4(l+1)ab^{-2}>1$. Therefore, $(k,l,a,b)\in \mathcal{S}_1$ or $(k,l,a,b)\in \mathcal{S}_2$.
  Consider $b=0$. In this case, $f(x)=x$ and $n=1$.  We can obtain that $k>1$ or  $k=1$, $8a>1$ from statement {\bf (\romannumeral3)}. Thus $(k,l,a,b)\in \mathcal{S}_3$ or $(k,l,a,b)\in \mathcal{S}_4$. In summary, if system \eqref{2m+1} has a global center at the origin, the parameters   belong to any one of spaces $\mathcal{S}_1,\ldots,\mathcal{S}_4$.
 The proof is complete.
 $\hfill{} \Box$

\section{Remarking conclusions}
\label{rc}
By Theorem \ref{thm3}, we can obtain that all orbits are bounded of system \eqref{1}  when condition {\bf (\romannumeral4)}  of Theorem \ref{thm1} holds.
Naturally, we want to know the properties of boundedness of orbits of system \eqref{1} when condition {\bf (\romannumeral4)}  of  Theorem \ref{thm1} does not hold.
First, we give the following lemma.
\begin{lemma}
	Assume that conditions {\bf (\romannumeral1)} and {\bf (\romannumeral3)}  of {\rm  Theorem \ref{thm1}} hold. 
	there is a value $\hat w\geq0$ such that $F_1(\hat w)=F_2(\hat w)$ and either  $F_1(w)<F_2(w)$ or $F_1(w)> F_2(w)$ for $w>\hat w$
when condition {\bf (\romannumeral4)}  of {\rm Theorem \ref{thm1}} does not hold,
where $w(x)$, $F_1(w)$ and $F_2(w)$ are still defined in Section \ref{pre}.
\label{c4}
\end{lemma}
\begin{proof}
It is obvious that $F_1(w)\not\equiv F_2(w)$ 
when condition {\bf (\romannumeral4)}  of Theorem \ref{thm1} does not hold. Then, we can assume that $\hat w\geq 0$ is the largest value  satisfying $F_1( w) =F_2( w)$ since $F(x)$ and $G(x)$ are polynomial. In other words, we have either  $F_1(w)<F_2(w)$ or $F_1(w)> F_2(w)$ for $w>\hat w$. Otherwise, there are two values $w_4>\hat w$ and $w_5>\hat w$ satisfying $F_1(w_4)<F_2(w_4)$ and $F_1(w_5)>F_2(w_5)$. On the one hand, we can find a value $w_6\in (\min\{w_4,w_5\}, \max\{w_4,w_5\})$ satifying  $F_1(w_6) =F_2(w_6)$ by intermediate value theorem.  On the other hand, $\hat w\geq 0$ is  the largest value  satisfying $F_1( w) =F_2(w)$. This is a contradiction.
The proof is finished.
\end{proof}

The following proposition will show an interesting result   when condition {\bf (\romannumeral4)}  of Theorem \ref{thm1} does not hold. Naturally,  by Lemma \ref{c4}, 	there is a value $\hat w\geq0$ such that $F_1(\hat w)=F_2(\hat w)$ and either  $F_1(w)<F_2(w)$ or $F_1(w)> F_2(w)$ for $w>\hat w$. 
\begin{proposition}
	\label{pro10}
Assume that conditions {\bf (\romannumeral1)} and {\bf (\romannumeral3)}  of {\rm  Theorem \ref{thm1}} hold. Then, all orbits of system \eqref{1} are positive bounded $($resp. negative bounded$)$ if and only if there  exists a value $\hat w\geq0$ such that $F_1(w)<($resp. $>)
F_2(w)$ for all $w>\hat w$.
\end{proposition}
\begin{proof}
On the one hand,
considering condition {\bf (\romannumeral1)}   of {\rm  Theorem \ref{thm1}}, system \eqref{1} has a unique equilibrium, which is an anti-saddle  by Lemma \ref{lem5}.  On the other hand,
considering condition  {\bf (\romannumeral3)}  of {\rm  Theorem \ref{thm1}}, system \eqref{1} has   no orbits connecting the equilibria at infinity in the Poincar\'e disc by Lemma  \ref{lem7}. 
In what follows,  for  convenience, we consider system \eqref{2}  because it is equivalent to system \eqref{1}.

Firstly, we show the sufficiency of this proposition.
  It suffices to prove this proposition in the case $F_1(w)<F_2(w)$ and  we can similarly obtain the
  results in the case $F_1(w)>F_2(w)$.

  Consider an
  orbit segment of system \eqref{2} with a point $A$ in the positive $y$-axis, where the ordinate $y_A$ of $A$ is sufficiently large.  It follows from nonexistence of orbits connecting the equilibria at infinity that this orbit must  cross the positive $x$-axis, intersect with the negative $y$-axis at a point $B: (0,y_B)$  and then intersect with the positive $y$-axis again at a point $C$, as shown in Figure \ref{tu1}(a). We can choose a point $A$ such that  $|y_B|$ is also sufficiently large, where $y_A$ is  large enough.
  
  	\begin{figure}[hpt]
  	\centering
  	{ 			\includegraphics[width=0.75\textwidth]{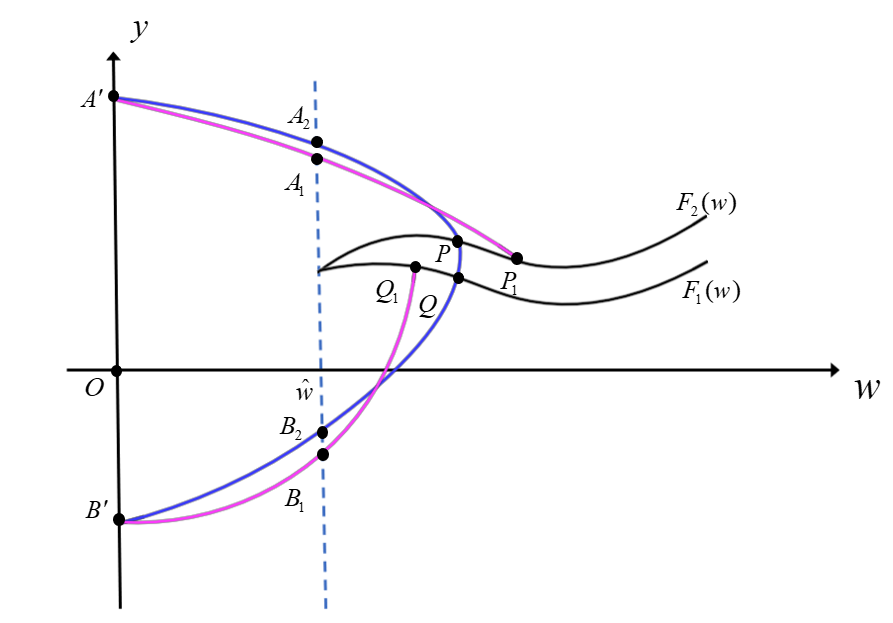}}
  	\caption{Orbits in $w$-$y$ plane. }
  	\label{tu13}
  \end{figure}
By the transformation \eqref{z},
let  the  orbit segment  $\widehat{ABC}\bigcap\{(x,y)|x\geq0\}$ be corresponded to $\widehat{A'PB'}$ in $w$-$y$ plane, where $P: (w_{P}, y_{P})$ lies on $y=F_2(w)$ and $A_2:(\hat w, y_{A_2})$, $B_2:(\hat w, y_{B_2})$ are  two points on $w=\hat w$, 
 as shown in Figure \ref{tu13}.
 
Consider an  integral curve  $\widehat{A'A_1P_1}$ of \eqref{z1}, where $A_1:(\hat{w}, y_{A_1})$ 
 lies on $w=\hat w$ and $P_1: (w_{P_1}, y_{P_1})$ lies on $y=F_2(w)$. We claim that $w_{P_1} >w_P$.
 In fact, 
 since integral curve $\widehat{A'A_2}$ (resp. $\widehat{A'A_1}$) satisfies  equation  \eqref{z2} (resp. \eqref{z1}), 
 we can obtain that 
 $$y_{A_2}-y_{A'}=\int_{0}^{\hat w}\frac{w^r}{F_2(w)-y_2(w)}dw \text{ (resp.  }
 y_{A_1}-y_{A'}=\int_{0}^{\hat w}\frac{w^r}{F_1(w)-y_1(w)}dw \text{)},$$
 where $y_2(w)$ (resp. $y_1(w)$) represents the integral curve $\widehat{A'A_2P}$ (resp. $\widehat{A'A_1P_1}$). Since $y_{A'}=y_A$ is sufficiently large, then there exists a sufficiently  small constant $\varepsilon>0$ such that 
 $$y_{A_2}-y_{A'}<\varepsilon/2$$ and 
  $$y_{A_1}-y_{A'}<\varepsilon/2.$$\
  Thus, we can obtain $$|y_{A_2}-y_{A_1}|<\varepsilon.$$
  
  Let $\varphi(w)=y_1(w)-y_2(w)$. Thus, $|\varphi(\hat w)|=|y_{A_2}-y_{A_1}|<\varepsilon$. For $w\in(\hat w,\min\{w_P,w_{P_1}\})$,
 it follows from  that 
     \begin{equation}
     	\label{12}
 \begin{aligned}
   	 \varphi(w)=&\int_{\hat w}^{w}\frac{z^r}{F_1(z)-y_1(z)}dz-\int_{\hat w}^{w}\frac{z^r}{F_2(z)-y_2(z)}dz+\varphi(\hat w)\\
   	 =&\int_{\hat w}^{w} \frac{z^r(F_2(z)-F_1(z))}{(F_1(z)-y_1(z))(F_2(z)-y_2(z))}dz\\
   	&\qquad + \int_{\hat w}^{w}\frac{z^r(y_1(z)-y_2(z))}{(F_1(z)-y_1(z))(F_2(z)-y_2(z))}dz
   	 +\varphi(\hat w)\\
   	 =:&H_1(w)+\int_{\hat w}^{w}H_2(z)\varphi(z)dz,
   	\end{aligned}
   \end{equation}
  where 
  $$
  H_1(w)=\varphi(\hat w)+\int_{\hat w}^{w} \frac{z^r(F_2(z)-F_1(z))}{(F_1(z)-y_1(z))(F_2(z)-y_2(z))}dz
  $$ 
  and
   $$
   H_2(w)=\frac{w^r}{(F_1(w)-y_1(w))(F_2(w)-y_2(w))}.
   $$
 Letting $H(w)=\int_{\hat w}^{w}H_2(z)\varphi(z)dz$, it follows from \eqref{12} that 
 \begin{equation}
 \label{H}
 \frac{dH(w)}{dw}=H_2(w)\varphi(w)=H_2(w)H_1(w)+H_2(w)H(w).
\end{equation} 
Using the  constant variation formula to solve equation \eqref{H}, we obtain that
 $$
 H(w)=\int_{\hat w}^{w}H_2(z)H_1(z)\exp\left(\int_{z}^{w}H_2(s)ds\right)dz.
 $$
  Thus, 
   \begin{equation}
  \notag
  	\begin{aligned}
  \varphi(w)=&H_1(w)+H(w)\\
  =&H_1(w)+\int_{\hat w}^{w}H_2(z)H_1(z)\exp\left(\int_{z}^{w}H_2(s)ds\right)dz\\
  =&H_1(\hat w)\exp\left(\int_{\hat w}^{w}H_2(s)ds\right)+\int_{\hat w}^{w}H'_1(z)\exp\left(\int_{z}^{w}H_2(s)ds\right)dz\\
  =&\varphi(\hat w)\exp\left(\int_{\hat w}^{w}H_2(s)ds\right)+\int_{\hat w}^{w}H'_1(z)\exp\left(\int_{z}^{w}H_2(s)ds\right)dz.
   	\end{aligned}
\end{equation}
  Since 
  $$
  H'_1(w)=\frac{w^r(F_2(w)-F_1(w))}{(F_1(w)-y_1(w))(F_2(w)-y_2(w))}>0
  $$ 
  for 
  $w\in (\hat w,  \min \{w_P, w_{P_1}\})$,
   we can have that 
  $$
  \int_{\hat w}^{w}H'_1(z)\exp\left(\int_{z}^{w}H_2(s)ds\right)dz>0.
  $$
  Moreover, $|\varphi(\hat w)|<\varepsilon$. One can check that 
  $\varphi(\w w)>0$, where $\w w= \min \{w_P, w_{P_1}\}$. 
  Consequently,  we have $w_{P_1} >w_P$.
  
  Consider $\widehat{B'B_1Q_1}$, the  integral curve starting from $B'$  of \eqref{z1}.  We can similarly prove $w_{Q_1}<w_Q$, where $Q_1$,  $Q$ are lie on $y=F_1(w)$ and $w_{Q_1}, w_Q$ are respectively abscissas of $Q_1$, $Q$.  Returning to the $x$-$y$ plane of system \eqref{2}, we get that its all orbits are positive-bounded. The proof of the sufficiency is done.
  
  Next, we consider the necessity. Similarly, we only need to prove that there 
exists a value $\hat w\geq0$ such that $F_1(w)<
 F_2(w)$ for all $w>\hat w$ when all orbits of system \eqref{1} are positive bounded. We choose $\hat w$ to be the largest   value  satisfying $F_1( w) =F_2( w)$. Using contradiction, according to the proof Lemma \ref{c4},  $F_1(w)>
 F_2(w)$ for all $w>\hat w$  is the inverse of $F_1(w)<
 F_2(w)$ for all $w>\hat w$. As proven above, we can obtain that all orbits of system \eqref{1} are negative bounded if there exists a value $\hat w\geq0$ such that $F_1(w)>
 F_2(w)$ for all $w>\hat w$, which is a contradiction. The proof of  the necessity is finished, so is this proposition.
\end{proof}
	\section*{Appendix }
	For the sake of completeness, we will show how to obtain the qualitative properties of system \eqref{m} at infinity in following two cases: (C1) $m=2n+1$, $\epsilon>(4n+4)^{-1}$, (C2) $m>2n+1$, $m$ is odd, $\epsilon=1$ in the appendix.
	
	Consider the case (C1).
 By a Poincar\'e transformation
 $$
 x=\frac{1}{z},\quad y=\frac{u}{z},
 $$
 system \eqref{m} is changed to
 
 \begin{equation}
 \label{x}
 \left\{\begin{aligned}
 \frac{du}{d\tau}=&\epsilon+u^2z^{2n}+\sum_{i=r}^{2n}\widehat{a}_iz^{2n+1-i}+uz^n+\sum_{i=s}^{n-1}\widehat{b}_iuz^{2n-i},\\
 \frac{dz}{d\tau}=&uz^{2n+1},
 \end{aligned}
 \right.
 \end{equation}
 where $d\tau = -dt/z^{2n}$. 
 It is easy to check that system \eqref{x} has no equilibria on $z=0$. 
 
 With the other Poincar\'e transformation   
  $$
  x=\frac{v}{z},\quad y=\frac{1}{z},
  $$
  system \eqref{m} is written as
 \begin{equation}
 \label{vz}
 \left\{\begin{aligned}
 \frac{dv}{d\tau}=&z^{2n}+\epsilon v^{2n+2}+\sum_{i=r}^{2n}\widehat{a}_iv^{i+1}z^{2n+1-i}+v^{n+1}z^{n}+\sum_{i=s}^{n-1}\widehat{b}_iv^{i+1}z^{2n-i},\\
 \frac{dz}{d\tau}=&\epsilon v^{2n+1}z+\sum_{i=r}^{2n}\widehat{a}_iv^{i}z^{2n+2-i}+v^{n}z^{n+1}+\sum_{i=s}^{n-1}\widehat{b}_iv^{i}z^{2n+1-i},
 \end{aligned}
 \right.
 \end{equation}
 where $d\tau = dt/z^{2n}$.
Notice that system \eqref{vz} has a unique equilibrium $B:(0,0)$ on $z=0$.
Furtherly, we can obtain that $B$ is a degenerate equilibrium. 
	With a polar transformation $(v,z)=(r\cos\theta,r\sin\theta)$,   system \eqref{vz} can be written as
	\begin{equation}
	\notag
	\frac{1}{r}\frac{dr}{d\theta}=\frac{H_1(\theta)+O(r)}{G_1(\theta)+O(r)},
	\end{equation}
	where 
	$G_1(\theta)=-\sin^{2n+1}\theta$ and $H_1(\theta)=\cos\theta\sin^{2n}\theta$.
		A necessary condition on existence of exceptional directions is $G_1(\theta)=0$ by \cite[Chapter 2]{ZDHD}.  Obviously,  $G_1(\theta)=0$   has exactly two roots $0$, $\pi$ in $\theta\in[0,2\pi)$.
		However,  $H_1(0)=H_1(\pi)=0$. Thus, we 
		cannot apply  the normal sector method  (see \cite[Chapter 2]{ZDHD}) to analyze the two exceptional directions
			$\theta=0$, $\pi$  of $B$  for system \eqref{vz}.  Instead, we adopt Briot--Bouquet transformations to blow up the two directions.
			
				With the Briot--Bouquet transformation
				$	z=\widetilde{z} v$, 
				system \eqref{vz} is changed into
 \begin{equation}
 \label{wz}
 \left\{\begin{aligned}
 \frac{dv}{d\delta}=&v{\w {z}}^{2n}+\epsilon v^{3}+\sum_{i=r}^{2n}\widehat{a}_iv^{3}{\w z}^{2n+1-i}+v^{2}{\w z}^{n}+\sum_{i=s}^{n-1}\widehat{b}_iv^{2}{\w z}^{2n-i},\\
 \frac{d\w {z}}{d\delta}=&-{\w z}^{2n+1},
 \end{aligned}
 \right.
 \end{equation}				
						where $d\delta=v^{2n-1}d\tau$.
System \eqref{wz} has a unique equilibrium 	$E: (0,0)$, which is degenerate. 
	The polar change of variables $(v,\w z)=(r \cos\theta,r\sin\theta)$ sends system \eqref{wz} to
	\begin{equation}
	\notag
	\frac{1}{r}\frac{dr}{d\theta}=\frac{H_2(\theta)+O(r)}{G_2(\theta)+O(r)},
	\end{equation}
	where
	$$
	G_2(\theta)=
	\left\{\begin{aligned}
&	-\sin\theta\cos\theta(2\sin^2\theta+\sin\theta\cos\theta+\epsilon\cos^2\theta), \qquad &\text{ if } n=1,\\
&	-\epsilon\sin\theta\cos^3\theta,\qquad &\text{ if } n>1,\\
 \end{aligned}
 \right.
 $$
	and
	$$
	H_2(\theta)=
		\left\{\begin{aligned}
	&	-\sin^4\theta+\sin^2\theta\cos^2\theta+\sin\theta\cos^3\theta+\epsilon\cos^4\theta, \qquad &\text{ if } n=1,\\
	&	\epsilon\cos^4\theta,\qquad &\text{ if } n>1.\\
	\end{aligned}
	\right.
	$$
	
	Firstly, consider $n=1$.  Since $\epsilon>1/8$, the equation 
	$ 2\sin^2\theta+\sin\theta\cos\theta+\epsilon\cos^2\theta=0$ has no real roots. Thus, $G_2(\theta)$ has four simple zeros $\theta=0$,  $\pi/2$,   $\pi$, $3\pi/2$  in $[0,2\pi)$. Moreover, we have
	 $$
	 G_2'(0)H_2(0)=G_2'(\pi)H_2(\pi)=-\epsilon^2<0
	 $$
	  and 
	  $$
	  G_2'(\pi/2)H_2(\pi/2)=G_2'(3\pi/2)H_2(3\pi/2)=-2<0.
	  $$
	   Thus, by 
	   $$
	   H_2(0)=H_2(\pi)=\epsilon>0, H_2(\pi/2)=H_2(3\pi/2)=-1<0
	   $$
	    and \cite[Theorem 3.7   of Chapter 2]{ZDHD}, system \eqref{wz} has
	a unique orbit  connecting $E$ in respectively the directions $\theta = \pi/2$ and $3\pi/2$  as $\delta\to+\infty$, and 	a unique orbit  connecting $E$ in respectively the
	directions $\theta= 0$ and $\pi$  as $\delta\to-\infty$, 
	see Figure \ref{tu9}(a).  Therefore, blowing down the equilibrium $E$ of system \eqref{wz} to $B$ of system \eqref{vz} yields that there is a unique orbit connecting $B$ in  the
	direction $\theta= 0$  as $\tau\to-\infty$, and  a unique orbit connecting $B$ in  the
	direction $\pi$ as $\tau\to+\infty$,	see Figure \ref{tu9}(b).
	
		\begin{figure}[hpt]
			
			\centering
			\subfigure[$v$-$\w z$ plane for system \eqref{wz}]{
				\includegraphics[width=0.45\textwidth]{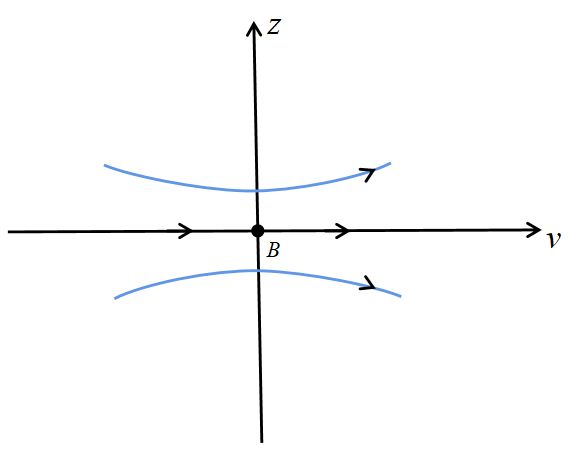}}\quad
			\subfigure[$v$-$z$ plane for system \eqref{vz}]{
				\includegraphics[width=0.45\textwidth]{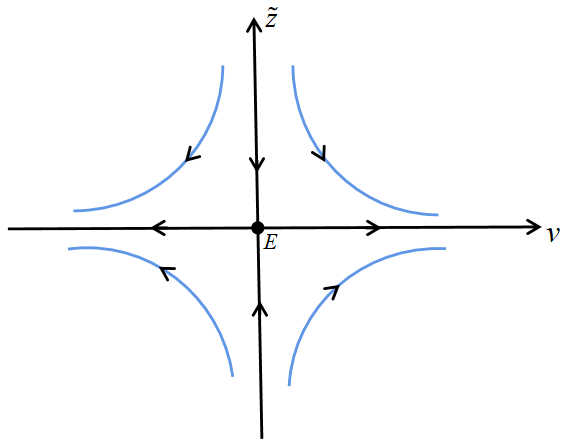}}\quad
			\caption{ Orbits changing under the Briot-Bouquet transformation for $n=1$.}
			\label{tu9}
		\end{figure}

	Then, we consider $n>1$.  It is easy to check that $G_2(\theta)$ has exactly two simple zeros  $0$, $\pi$ and two trible zeros $\pi/2$, $3\pi/2$ in $[0,2\pi)$.  Easy calculation gives that $G_2'(0)H_2(0)=G_2'(\pi)H_2(\pi)=-\epsilon^2<0$. One can obtain that there is  a unique orbit connecting $E$ in respectively the
	directions $\theta= 0$ and $\pi$  as $\delta\to-\infty$ by $H_2(0)=H_2(\pi)=\epsilon>0$ and \cite[Theorem 3.7   of Chapter 2]{ZDHD}.  However, since $H_2(\pi/2)=H_2(3\pi/2)=0$, we need to blow up the two directions $\theta=\pi/2$, $3\pi/2$.
	
	With the quasihomogeneous blow-up  $v={\w z}^{n-1}{\w v}$, (see \cite[Chapter 3.3]{DLA}, or \cite{Dumortier}), we change system
\eqref{wz} into 
\begin{equation}
\label{wv}
\left\{\begin{aligned}
\frac{d\w v}{ds}=&n{\w v}{\w {z}}^{2}+\epsilon {\w v}^{3}+\sum_{i=r}^{2n}\widehat{a}_i{\w v}^{3}{\w z}^{2n+1-i}+{\w v}^{2}{\w z}+\sum_{i=s}^{n-1}\widehat{b}_i{\w v}^{2}{\w z}^{n+1-i},\\
\frac{d\w {z}}{ds}=&-{\w z}^3,
\end{aligned}
\right.
\end{equation}	
	where $ds={\w z}^{2n-2}d\delta$. System \eqref{wv} has a unique equilibrium $F: (0,0)$. Similarly,  transforming system \eqref{wv} into equation
		\begin{equation}
		\notag
		\frac{1}{r}\frac{dr}{d\theta}=\frac{H_3(\theta)+O(r)}{G_3(\theta)+O(r)},
		\end{equation}
by a polar transformation $(\w v,\w z)=(r\cos\theta, r\sin\theta)$, where
$$
	G_3(\theta)=
		-\sin\theta\cos\theta((n+1)\sin^2\theta+\sin\theta\cos\theta+\epsilon\cos^2\theta)
		$$
	and 
	$$
	H_3(\theta)	=
-\sin^4\theta+n\sin^2\theta\cos^2\theta+\sin\theta\cos^3\theta+\epsilon\cos^4\theta.
$$
It follows from $\epsilon>1/(4n+4)$ that 
$$
(n+1)\sin^2\theta+\sin\theta\cos\theta+\epsilon\cos^2\theta=0
$$
 has no roots. Thus, $G_3(\theta)$ has four simple zeros $\theta=0$,  $\pi/2$,   $\pi$, $3\pi/2$  in $[0,2\pi)$.  Moreover, we can check that 
$$
G_3'(0)H_3(0)=G_3'(\pi)H_3(\pi)=-\epsilon^2<0
$$ 
and 
$$
G_3'(\pi/2)H_3(\pi/2)=G_3'(3\pi/2)H_3(3\pi/2)=-(n+1)<0.
$$ 
By 
$$
H_3(0)=H_3(\pi)=\epsilon>0, H_3(\pi/2)=H_3(3\pi/2)=-1<0
$$
 and \cite[Theorem 3.7   of Chapter 2]{ZDHD}, there is a unique orbit  connecting $F$ in respectively the directions $\theta = \pi/2$ and $3\pi/2$  as $s\to+\infty$, and 	a unique orbit  connecting $F$ in respectively the
directions $\theta= 0$ and $\pi$  as $s\to-\infty$, see Figure \ref{tu10}(a). Furtherly, we can obtain that there is a unique orbit  connecting $E$ in respectively the
directions $\theta= \pi/2$ and $3\pi/2$  as $\delta\to+\infty$. The qualitative properties of $E$ are  shown in Figure \ref{tu10}(b). Then, by blowing down the equilibrium $E$ of system \eqref{wz} to $B$ of system \eqref{vz},   there is a unique orbit connecting $B$ in  the
direction $\theta= 0$  as $\tau\to-\infty$, and  a unique orbit connecting $B$ in  the
direction $\pi$ as $\tau\to+\infty$,	see Figure \ref{tu10}(c).

	\begin{figure}[hpt]
		
		\centering
		\subfigure[$\w v$-$\w z$ plane for system \eqref{wv}]{
			\includegraphics[width=0.45\textwidth]{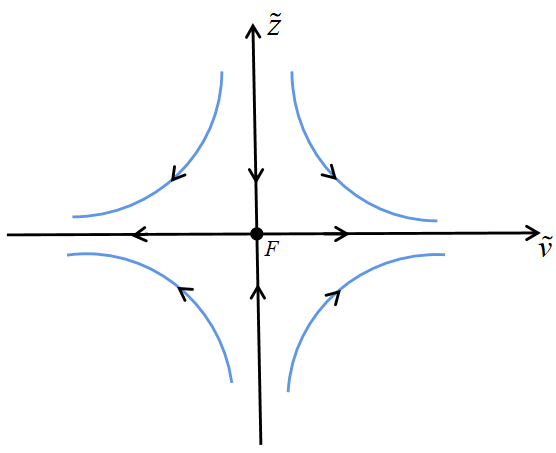}}\quad
		\subfigure[$v$-$\w z$ plane for system \eqref{wz}]{
			\includegraphics[width=0.45\textwidth]{6}}\quad

		\subfigure[$v$-$ z$ plane for system \eqref{vz}]{
			\includegraphics[width=0.45\textwidth]{7}}\quad
		\caption{ Orbits changing under the Briot-Bouquet transformation for $n>1$.}
		\label{tu10}
	\end{figure}
	
		Consider the case (C2). 
Similarly		by a Poincar\'e transformation
		$$
		x=\frac{1}{z},\quad y=\frac{u}{z},
		$$
		system \eqref{m} is changed to
		
		\begin{equation}
		\label{y}
		\left\{\begin{aligned}
		\frac{du}{d\tau}=&1+u^2z^{m-1}+\sum_{i=r}^{m-1}\widehat{a}_iz^{m-1-i}+uz^{m-1-n}+\sum_{i=s}^{n-1}\widehat{b}_iuz^{m-1-i},\\
		\frac{dz}{d\tau}=&uz^{m},
		\end{aligned}
		\right.
		\end{equation}
where $d\tau = -dt/z^{m-1}$. There is 
 no equilibrium on $z=0$ of system \eqref{y}. 

By another Poincar\'e transformation   
$$
x=\frac{v}{z},\quad y=\frac{1}{z},
$$
system \eqref{m} is written as
\begin{equation}
\label{13}
\left\{\begin{aligned}
\frac{dv}{d\tau}=&z^{m-1}+ v^{m+1}+\sum_{i=r}^{m-1}\widehat{a}_iv^{i+1}z^{m-i}+v^{n+1}z^{m-1-n}+\sum_{i=s}^{n-1}\widehat{b}_iv^{i+1}z^{m-1-i},\\
\frac{dz}{d\tau}=& v^{m}z+\sum_{i=r}^{2n}\widehat{a}_iv^{i}z^{m+1-i}+v^{n}z^{m-n}+\sum_{i=s}^{n-1}\widehat{b}_iv^{i}z^{m-i},
\end{aligned}
\right.
\end{equation}
where $d\tau = dt/z^{m-1}$. Then, $B: (0,0)$ is the unique equilibrium of system \eqref{13} on $z=0$. Considering a polar transformation $(v,z)=(r\cos\theta,r\sin\theta)$,   system \eqref{vz} can be written as
\begin{equation}
\notag
\frac{1}{r}\frac{dr}{d\theta}=\frac{H_4(\theta)+O(r)}{G_4(\theta)+O(r)},
\end{equation}
where 
$G_4(\theta)=-\sin^{m}\theta$ and $H_4(\theta)=\cos\theta\sin^{m-1}\theta$. It is easy to check that
$G_4(\theta)=0$   has  two roots $0$ and $\pi$ in $\theta\in[0,2\pi)$. However, since $H(0)=H(\pi)=0$, we need to desingularize $B$ furtherly. 

With 	the Briot--Bouquet transformation
$	z=\widetilde{z} v$, 
system \eqref{13} is rewritten as 
 \begin{equation}
 \label{14}
 \left\{\begin{aligned}
 \frac{dv}{d\delta}=&v{\w {z}}^{m-1}+ v^{3}+\sum_{i=r}^{m-1}\widehat{a}_iv^{3}{\w z}^{m-i}+v^{2}{\w z}^{m-1-n}+\sum_{i=s}^{n-1}\widehat{b}_iv^{2}{\w z}^{m-1-i},\\
 \frac{d\w {z}}{d\delta}=&-{\w z}^{m},
 \end{aligned}
 \right.
 \end{equation}	
	where $d\delta=v^{m-2}d\tau$. Obviously, $E:(0,0)$ of system \eqref{14} is the equilibrium. It follows from  $m>2n+1$ and $n\geq1$ that $m>3$.	Using a polar coordinate $(v,\w z)=(r\cos\theta, r\sin\theta)$,
	system \eqref{14} is transformed into the following polar form
	\begin{equation}
	\notag
	\frac{1}{r}\frac{dr}{d\theta}=\frac{H_5(\theta)+O(r)}{G_5(\theta)+O(r)},
	\end{equation}
	$
	G_5(\theta)=
-\sin\theta\cos^3\theta$
	and
	$
	H_5(\theta)=
	\cos^4\theta$.  
 $G_5(\theta)$ has two simple zeros $\theta=0$,  $\pi$ and two triple zeros   $\pi/2$, $3\pi/2$  in $[0,2\pi)$. Moreover, we can check that 
 $$
 G_5'(0)H_5(0)=G_5'(\pi)H_5(\pi)=-1<0.
 $$
  Thus, by $H_5(0)=H_2(\pi)=1$ and \cite[Theorem 3.7   of Chapter 2]{ZDHD}, system \eqref{14} has
 a unique orbit  connecting $E$ in respectively the directions $\theta= 0$ and $\pi$  as $\delta\to-\infty$. However, since $H_5(\pi/2)=H_2(3\pi/2)=0$, we need to blow up the two directions.
 
 The quasihomogeneous blow-up $v=\w z^{(m-3)/2}\w v$ sends sytem \eqref{14} to

\begin{equation}
\label{19}
\left\{\begin{aligned}
\frac{d\w v}{ds}=&{\frac{m-1}{2}}{\w v}{\w {z}}^{2}+ {\w v}^{3}+\sum_{i=r}^{m-1}\widehat{a}_i{\w v}^{3}{\w z}^{m-i}+{\w v}^{2}{\w z}^{\frac{m-2n+1}{2}}+\sum_{i=s}^{n-1}\widehat{b}_i{\w v}^{2}{\w z}^{\frac{m-2i+1}{2}},\\
\frac{d\w {z}}{ds}=&-{\w z}^3,
\end{aligned}
\right.
\end{equation}	
where $ds={\w z}^{m-3}d\delta$. It is clear that $F:(0,0)$ is an equilibrium.  Using a polar coordinate $(\w v,\w z)=(r\cos\theta, r\sin\theta)$, system \eqref{19} is transformed into the following polar
form 
\begin{equation}
\notag
\frac{1}{r}\frac{dr}{d\theta}=\frac{H_6(\theta)+O(r)}{G_6(\theta)+O(r)},
\end{equation} 
where 
$$	G_6(\theta)=
	-\sin\theta\cos\theta\left(\frac{m+1}{2}\sin^2\theta+\cos^2\theta\right)
	$$
	and 
	$$
	H_6(\theta)	=
	-\sin^4\theta+\frac{m-1}{2}\sin^2\theta\cos^2\theta+\cos^4\theta.
	$$	
	$\theta=0$,  $\pi/2$,   $\pi$, $3\pi/2$  are four simple zeros of
 $G_6(\theta)$   in $[0,2\pi)$. Moreover,  
 $$
 G_6'(0)H_6(0)=G_6'(\pi)H_6(\pi)=-1<0
 $$ 
 and 
 $$
 G_6'(\pi/2)H_6(\pi/2)=G_6'(3\pi/2)H_6(3\pi/2)=-(m+1)/2<0.
 $$ 
 By 
 $$
 H_6(0)=H_6(\pi)=1>0, H_3(\pi/2)=H_3(3\pi/2)=-1<0
 $$
  and \cite[Theorem 3.7   of Chapter 2]{ZDHD}, there is a unique orbit  connecting $F$ in respectively the directions $\theta = \pi/2$ and $3\pi/2$  as $s\to+\infty$, and 	a unique orbit  connecting $F$ in respectively the
 directions $\theta= 0$ and $\pi$  as $s\to-\infty$, as also seen  in Figure \ref{tu10}(a). Furtherly, we can obtain that there is a unique orbit  connecting $E$ in respectively the
 directions $\theta= \pi/2$ and $3\pi/2$  as $\delta\to+\infty$. The qualitative properties of $E$ are also  shown in Figure \ref{tu10}(b). Then, by blowing down the equilibrium $E$ of system \eqref{wz} to $B$ of system \eqref{vz},   there is a unique orbit connecting $B$ in  the
 direction $\theta= 0$  as $\tau\to-\infty$, and  a unique orbit connecting $B$ in  the
 direction $\pi$ as $\tau\to+\infty$,	see Figure \ref{tu10}(c).

		\section*{Acknowledgements}
This paper is supported by the  National Natural Science Foundation of China (No.   12171485).

\end{document}